\documentclass[12pt, twoside]{article}
\usepackage{amsmath,amsthm,amssymb,color}
\usepackage{times}
\usepackage{enumerate}
\usepackage[utf8]{inputenc}
\usepackage[T1]{fontenc}

\def\titlerunning#1{\gdef\titrun{#1}}
\makeatletter
\def\author#1{\gdef\autrun{\def\and{\unskip, }#1}\gdef\@author{#1}}
\def\address#1{{\def\and{\\\hspace*{18pt}}\renewcommand{\thefootnote}{}%
\footnote {#1}}%
\markboth{\autrun}{\titrun}}
\makeatother
\def\email#1{e-mail: #1}
\def\subjclass#1{{\renewcommand{\thefootnote}{}%
\footnote{\emph{Mathematics Subject Classification (2010):} #1}}}

\newtheorem{theorem}{Theorem}[section]
\newtheorem{corollary}[theorem]{Corollary}
\newtheorem{lemma}[theorem]{Lemma}
\newtheorem{hyp}[theorem]{Hypothesis}

\makeatletter
\newcommand{\tpmod}[1]{{\@displayfalse\pmod{#1}}}
\makeatother

\theoremstyle{definition}

\newtheorem*{remark}{Remark}

\numberwithin{equation}{section}

\newcommand\numberthis{\addtocounter{equation}{1}\tag{\theequation}}

\begin{document}


\baselineskip=17pt


\titlerunning{Asymptotic trace formula}

\title{Asymptotic trace formula for the Hecke operators}

\author{Junehyuk Jung
\and
Naser T. Sardari
\and (With an appendix by Simon  Marshall)}
\date{}
\maketitle
\address{J. Jung: Department of Mathematics, Brown University, Providence, RI 02912 USA; \email{junehyuk\textunderscore jung@brown.edu}
\and N. T. Sardari: The Institute
For Advanced Study, Princeton, NJ 08540 USA;  \email{ntalebiz@ias.edu}
\and S. Marshall: Department of Mathematics, UW-Madison, Madison, WI 53706 USA; \email{marshall@math.wisc.edu })}



\begin{abstract}
Given integers $m$, $n$ and $k$,  we  give  an explicit formula  with an optimal error term (with square root cancelation) for the Petersson trace formula involving the $m$th and $n$th Fourier coefficients of an orthonormal basis of $S_k\left(N\right)^*$ (the weight $k$ newforms with fixed square-free level $N$) provided that $|4 \pi \sqrt{mn}- k|=o\left(k^{\frac{1}{3}}\right)$. Moreover, we  establish an explicit formula with a power saving error term for  the trace of the  Hecke operator $\mathcal{T}_n^*$ on $S_k\left(N\right)^*$  averaged over   $k$ in a short interval.  By  bounding  the second moment of the trace of  $\mathcal{T}_{n}$ over a larger interval, we show that the trace of $\mathcal{T}_n$ is unusually large in the range  $|4 \pi \sqrt{n}- k| = o\left(n^{\frac{1}{6}}\right)$. As an application, for any fixed prime $p$ coprime to $N$,  we show that there exists a sequence $\{k_n\}$  of weights  such that the error term of Weyl's law for $\mathcal{T}_p$ is unusually large and  violates the prediction of arithmetic quantum chaos. In particular, this generalizes the result of  Gamburd, Jakobson and Sarnak~\cite[Theorem 1.4]{Gamburd} with an improved exponent.
\end{abstract}

\bigskip
\footnotesize
\noindent\textit{Acknowledgments.}
J.J. thanks S.M. and Department of Mathematics of UW-Madison for invitation and support. J.J. also thanks Sug Woo Shin, Peter Jaehyun Cho, and Matthew Young for many helpful comments.
J.J. was supported by NSF grant DMS-1900993, and by Sloan Research Fellowship. S.M. was supported by NSF grant DMS-1902173. N.T.S. was supported by NSF grant DMS-2015305 and  is grateful to Max Planck Institute for Mathematics in Bonn and  Institute
For Advanced Study for their hospitalities and financial supports.   N.T.S. thanks his Ph.D. advisor Peter Sarnak for several  insightful and inspiring conversations regarding the error term of the Weyl law while he was a graduate student at Princeton University.


\normalsize
\baselineskip=17pt

\section{Introduction}
In this paper, we give bounds for the error term of Weyl's law  for the Hecke eigenvalues of the family of classical  holomorphic modular forms with a fixed level. We briefly describe this family, its Weyl's law, and known bounds and  predictions on its error term. Next,  we explain our results and compare them with the previous  results and predictions.

Let
\[
\Gamma_0\left(N\right):=\left\{\begin{pmatrix}a & b \\ c &d  \end{pmatrix}\in SL_2\left(\mathbb{Z}\right)~:~  c\equiv 0   \pmod{N}   \right\}
\]  be the Hecke congruence subgroup of level $N$. Let  $S_k\left(N\right)$ be the space of even weight $k\in \mathbb{Z}$ modular forms of level $N$. It is the space of the holomorphic functions $f$ such that
\[
f\left(\frac{az+b}{cz+d}\right)=\left(cz+d\right)^k f\left(z\right)
\]
for every $\begin{pmatrix} a &b \\ c & d \end{pmatrix} \in \Gamma_0\left(N\right)$, and $f$ converges to zero as it approaches each cusp (we have finitely many cusps for $\Gamma_0\left(N\right)$ that are associated to the orbits of $\Gamma_0\left(N\right)$ acting by M\"obius transformations on $\mathbb{P}^1\left(\mathbb{Q}\right)$) \cite{sarnak_1990}.   It is well-known that $S_k\left(N\right)$ is a finite dimensional vector space over $\mathbb{C}$, and is equipped with the Petersson inner product
\[
\langle f, g \rangle:=\int_{\Gamma_0\left(N\right) \backslash \mathbb{H}} f\left(z\right)\overline{g\left(z\right)} y^k \frac{dx dy}{y^2}.
\]
 Assume that $n$ is fixed and  is coprime to $N$. Then the $n$th (normalized)  Hecke operator $\mathcal{T}_n$ acting on $S_k\left(N\right)$ is given by
\begin{equation}\label{Hecke}
\mathcal{T}_n\left(f\right)\left(z\right):=n^{\frac{k-1}{2}}\sum_{ad=n}d^{-k} \sum_{b \pmod{d}} f\left(\frac{az+b}{d} \right).
\end{equation}
The Hecke operators form a commuting family of self-adjoint operators with respect to the Petersson inner product, and therefore $S_k\left(N\right)$ admits an orthonormal basis $B_{k,N}$ consisting only of joint eigenfunctions of the Hecke operators. Any form $f$ belonging to $ B_{k,N}$ is referred as a (Petersson normalized) holomorphic Hecke cusp form.

Let
\[
f\left(z\right)=\sum_{n=1}^{\infty} \rho_f(n) n^{\frac{k-1}{2}}e\left(nz\right)
\]
be the Fourier expansion of $f \in B_{k,N}$ at the cusp $\infty$. For $\gcd(n,N)=1$, we denote by $\lambda_f(n)$ the $n$th (normalized) Hecke eigenvalue of $f$, i.e.,
\[
\mathcal{T}_n f = \lambda_f(n) f,
\]
and we have $\rho_f(n) = \rho_f(1) \lambda_f(n)$ \cite[p. 107]{Iwaniecb}. By the celebrated result due to Deligne  \cite{Deligne}, we have
\begin{equation}\label{deligne}
|\lambda_f\left(n\right)|\leq \sigma(n),
\end{equation}
for all $n$, where $\sigma$ is the divisor function. \footnote{We use the divisor function parameterized by $t$, defined by $\sigma_t\left(n\right) =\sum_{d|n} d^t$. When $t=0$, we drop $0$, and use $\sigma$ instead of $\sigma_0$.}

Under Langlands' philosophy,  the Hecke operator $\mathcal{T}_p$ is the $p$-adic analogue of the Laplace operator, in the following sense. The eigenvalues of $\mathcal{T}_p$ determine the Satake parameters of the associated local representation $\pi_{p}$ of $GL_2\left(\mathbb{Q}_p\right)$ just as the Laplace eigenvalue of the Maass form determines  the associated local representation $\pi_{\infty}$ of  $GL_2\left(\mathbb{R}\right)$.

Now fix a rational prime $p$ that is coprime to $N$, and let
\[
\mu_{k,N}:=\frac{1}{\dim\left(S_k\left(N\right)\right)} \sum_{ f\in B_{k,N}} \delta_{\lambda_f\left(p\right)}
\]
be the spectral probability measure associated to $\mathcal{T}_p$ acting on $S_k\left(N\right)$. Using the Eichler--Selberg trace formula, Serre \cite{serre} proved that $\mu_{k,N}$ converges weakly to $\mu_p$ as $k+N\to \infty$ with $\gcd(N,p)=1$, where $\mu_p$ is the Plancherel measure of $GL_2\left(\mathbb{Q}_p\right)$ given by
\[
\mu_p\left(x\right):=\frac{p+1}{\pi} \frac{\left(1-\frac{x^2}{4}\right)^{\frac{1}{2}}}{\left(p^{\frac{1}{2}}+p^{-\frac{1}{2}}\right)^2-x^2} dx.
\]
Moreover, if we let
\[
\nu_{k,N}:= \frac{\Gamma(k-1)}{(4\pi)^{k-1}}\sum_{f\in B_{k,N}} |\rho_f(1)|^2 \delta_{\lambda_f\left(p\right)},
\]
then it follows from the Petersson trace formula (see Section~\ref{peters})  that $\nu_{k,N}$ converges weakly to the semi-circle law
\[
\mu_{\infty}\left(x\right):=\frac{1}{\pi} \sqrt{1-\frac{x^2}{4}} dx,
\]
as $k+N\to \infty$ with $\gcd(N,p)=1$.

\subsection{Quantitative rate of convergence}
Given two probability measures $\mu_1$ and $\mu_2$ on $\mathbb{R}$, we denote the discrepancy between them by $D\left(\mu_1,\mu_2\right),$ where
$$
D\left(\mu_1,\mu_2\right):=\sup \{|\mu_1\left(I\right)-\mu_2\left(I\right)|: I=[a,b]\subset \mathbb{R}   \}.
$$
 In \cite{Gamburd},  Gamburd, Jakobson and Sarnak studied the spectrum of the elements in the group ring of $SU\left(2\right)$, and proved
\begin{equation}\label{upperweyl}
D\left(\mu_{k,2},\mu_p\right)=O\left(\frac{1}{\log k}\right).
\end{equation}
Moreover, \cite[Theorem 1.4]{Gamburd} is equivalent  to  the existence of  a sequence of integers  $k_n\to \infty$ such that
\begin{equation}\label{gap}
D\left(\mu_{k_n,2},\mu_p\right) \gg \frac{1}{k_n^{\frac{1}{2}}(\log k_n)^2}.
\end{equation}
This is a corollary of their lower bound for the variance of the trace of the Hecke operators by varying the weight $k$ (Theorem~\ref{selberg2}).

\subsection{Bounds for the error term of Weyl's law}

We now present some details regarding the philosophical analogy between Weyl's law and $\mu_{k,N} \to \mu_p$. To this end, we first review Weyl's law. Let $X\subset \mathbb{R}^d$ be a bounded domain with smooth boundary. Let $T$ be a positive real number, and let $N\left(T\right)$ be the number of Dirichlet Laplacian eigenvalues of $X$ less than $ T^2$ (counted with multiplicity). It was conjectured independently by Sommerfeld and Lorentz, based on the work of Rayleigh on the theory of sound, and proved by Weyl  \cite{Weyl1911} shortly after, that
$$
N\left(T\right)=c_d \text{vol}\left(X\right)T^{d} \left(1+o\left(1\right)\right) \quad  \text{ as } \lambda \to \infty,
$$
where $c_d$ is a constant depending only on $d$ and $\text{vol}\left(X\right)$ is the volume of $X$  in $\mathbb{R}^d$. This gives the distribution of the eigenvalues of the  Laplace--Beltrami operator as $T\to \infty$. As in Langlands' philosophy, this is analogous to the convergence of $\mu_{k,N}\to \mu_p$ giving the distribution of the Hecke eigenvalues as $k\to \infty.$

More generally, let $\left(M^d,g\right)$ be a compact smooth Riemannian manifold of dimension $d$. Let $N(T)$ be the number of eigenvalues of the Laplace--Beltrami operator $-\Delta_g$ less than $T^2$, counted with multiplicity. Then H\"ormander \cite{Lars} proved that
$$
N\left(T\right)= c_d \text{vol}\left(M\right)T^{d} +R_M\left(T\right),
$$
where $R_M\left(T\right)= O\left(T^{d-1}\right)$. In fact, this general estimate is sharp for the round sphere $M=S^d$. However, given a manifold $M$ the question of finding the optimal bound for the error term $R_M\left(T\right)$ is a very difficult problem. An analogue of $R_m(T)$ for $\mu_{k,N} \to \mu_p$ is the discrepancy $D(\mu_{k,N},\mu_p)$.

\begin{remark}
If $M$ is a symmetric space, then Weyl's law is formulated and expected to hold in great  generality for families of automorphic forms \cite[Conjecture 1]{Shin}.
\end{remark}

We now restrict to the case $d = 2$, and discuss the relation between the size of $R_M\left(T\right)$ and the geodesic flow on the unit cotangent bundle $S^*M$, predicted by the correspondence principle.  The two extreme behaviors that the geodesic flow can have are being chaotic or completely integrable, and in these two cases the correspondence principle predicts the distribution of eigenvalues to be modeled by a large random matrix, and a Poisson process, respectively \cite{Berry1,Berry2}.

In particular, we expect that for a generic $2$ dimensional flat torus or a compact arithmetic hyperbolic surface~\cite[Figure 1.3 and Section 3]{Arithmetic}\footnote{The geodesic flow in this case is chaotic, but Sarnak explains that one expects to see Poisson behavior due to the high multiplicity of the geodesic length spectrum.}, the set of eigenvalues inside the universal interval  $\left\lbrack T^2, \left(T+\frac{1}{L}\right)^2\right\rbrack$, where $\log T\ll L= o\left(T\right)$,\footnote{Here and elsewhere we write $A \ll_\tau B $ when $|A|\leq C(\tau)B$ holds with some constant $C(\tau)$ depending only on $\tau$.} is modeled  by Poisson process. For details, we refer the readers to the very interesting work of Rudnick \cite{Rudnick} and Sarnak's letter \cite{Sarnak3} explaining the critical window $\log T\ll L= o\left(T\right)$ using Kuznetsov's trace formula.  This suggests that these surfaces satisfy $R_M\left(T\right)= O_{\epsilon}\left(T^{\frac{1}{2}+\epsilon}\right)$.   In fact, Petridis and Toth proved that the average order of the error term in Weyl's law for a random torus chosen in a compact part of the moduli space of two dimensional tori is $R\left(T\right)=O_{\epsilon}\left(T^{\frac{1}{2}+\epsilon}\right)$ \cite{Toth}. Moreover, for compact arithmetic surfaces it was proved by Selberg \cite[p.315]{Hejhal} that $R\left(T\right)=\Omega\left(T^{\frac{1}{2}}/\log T\right)$. This bound is the analogue of~\eqref{gap}.

For the rational  torus $\mathbb{T}=\mathbb{R}^2/\mathbb{Z}^2$,  bounding $R_{\mathbb{T}}\left(T\right)$ is  equivalent to the classical  Gauss circle problem.  It was conjectured by Hardy that $R_{\mathbb{T}}\left(T\right)=O_{\epsilon}\left(T^{\frac{1}{2}+\epsilon}\right)$, and it is known by Hardy and Landau \cite{Hardy244} that $R_{\mathbb{T}}\left(T\right)=\Omega\left(T^{\frac{1}{2}} \left(\log T\right)^{\frac{1}{4}}\right)$. Note that the eigenvalue distribution here is known not to be Poisson \cite{Sarnaktori}.

As mentioned above, for generic compact hyperbolic surfaces, we expect the set of eigenvalues inside the interval $\left\lbrack T^2, \left(T+\frac{1}{L}\right)^2\right\rbrack$ to follow the eigenvalue distribution of a large symmetric matrix, which has a rigid structure. As a result, it is conjectured that these surfaces satisfy $R_M\left(T\right)=O_\epsilon \left(T^{\epsilon}\right)$.

Proving an optimal  upper bound for $R_M\left(T\right)$ is extremely difficult, and we do not have any explicit example of $M$ other than the sphere where the optimal bound is known!  The best known upper bound for hyperbolic manifolds is $R_M\left(T\right)=O\left(T^{d-1}/\log T\right)$, due to B\'erard \cite{Pierre}.  As pointed out by Sarnak \cite[p. 2]{Sarnak3}, even improving the constant and  showing that $R\left(T\right)=o\left(T/\log T\right)$ for the cuspidal spectrum  of $SL_2\left(\mathbb{Z}\right)\backslash \mathbb{H}$ (after removing the contribution of the Eisenstein series)  is  very difficult (Remark~\ref{errorm}). This bound is the analogue of~\eqref{upperweyl}.

\subsection{Main results}
\subsubsection{Large discrepancy for $\mu_{k,N}^*$}
Let $S_k\left(N\right)^*$ be the subspace of $S_k(N)$ consisting only of newforms of weight $k$ and fixed level $N$.  Let $\mathcal{T}_p^*$ be the restriction of $\mathcal{T}_p$ from $S_k\left(N\right)$ to $S_k\left(N\right)^*$. We denote by $\mu_{k,N}^*$ and $\nu_{k,N}^*$ the spectral probability  measures associated to $\mathcal{T}_p^*$, defined analogously to $\mu_{k,N}$ and $\nu_{k,N}$.

The main theorem of this paper is a generalization of \eqref{gap} to $\mu_{k,N}^*$ with any squarefree level $N$ and  an improved exponent of $k$ in the lower bound.

\begin{theorem}\label{rmwe}
Let $N\geq 1$ be a fixed square-free integer. Then there exists an infinite sequence of weights $\{ k_n\}$ with $k_n\to \infty$ such that
\[
D\left(\mu_{k_n,N}^*,\mu_p\right)\gg \frac{1}{k_n^{\frac{1}{3}}(\log k_n)^2}.
\]
\end{theorem}
\begin{remark}\label{errorm}
As mentioned in the introduction the best known upper bound for $D\left(\mu_{k,N}^*,\mu_p\right)$ is
\begin{equation}\label{upb1}
D\left(\mu_{k,N}^*,\mu_p\right) = O\left( \frac{1}{\log k}\right),
 \end{equation}
by Murty and Sinha \cite{MS}. The standard method for giving an upper bound for the discrepancy of a sequence of points  is the Erd\H{o}s--Tur\'an inequality~\cite{Erd}. Even to improve the implied constant in \eqref{upb1} using the Erd\H{o}s--Tur\'an inequality, one needs to obtain a nontrivial upper bound for the trace of the Hecke operator $\mathcal{T}_n$ for $n\gg k^A$,
where $A>0$ is an arbitrarily large constant. But the error term in the  Selberg trace formula is very hard to bound non-trivially in this range and this makes the problem very difficult by this approach.
\end{remark}

Theorem \ref{rmwe} follows from an explicit asymptotic formula for the weighted average of the trace of the  Hecke operator in a short interval. More precisely, let
 $\psi$ be a non-negative smooth function supported in $\left\lbrack-1,1\right\rbrack$ that satisfies $\int_{-1}^{1}\psi\left(t\right)dt=1$. Let $\mathrm{Tr}~\mathcal{T}_n\left(k,N\right)^*$ be the trace of the  Hecke operator $\mathcal{T}_n^*$ on $S_k\left(N\right)^*$.

\begin{theorem}\label{noweight}
Let $N\geq 1$ be a fixed square-free integer, and  $\frac{1}{5}<\delta<\frac{1}{3}$ be a fixed constant. Let $K$ be an integer satisfying $K=4\pi\sqrt{n}+o\left(n^{\frac{1}{6}}\right)$. Then we have
\begin{multline*}
\frac{1}{K^{\delta}} \sum_{k\in 2\mathbb{N}} \psi \left(\frac{k-K}{K^{\delta}}\right)(-1)^{\frac{k}{2}}\mathrm{Tr}~\mathcal{T}_n\left(k,N\right)^*\\
= \frac{\mu\left(N\right)K}{2\pi}\frac{\sigma\left(n\right)}{n} J_K\left(4\pi\sqrt{n}\right)  \left(1+o_{\delta,\psi}\left(1\right)\right),
\end{multline*}
 where $J_K$ is the $J$-Bessel function (the Bessel function of the first kind) and $\mu$ is the M\"obius function.
 \begin{remark}\label{tranrem}
 By the asymptotic of the $J$-Bessel function in the transition range \eqref{trans} (see also \S\ref{Bessel}), we have $|J_K\left(4\pi\sqrt{n}\right)|\gg K^{-\frac{1}{3}}$. Hence, we have $| \mathrm{Tr}~\mathcal{T}_n\left(k,N\right)^*| \gg k^{\frac{2}{3}}$ for some $k\in [K-K^{\delta}, K+K^{\delta}]$. This lower bound violates the naive expected square root cancelation for the  eigenvalues of the Hecke operator $\mathcal{T}_n\left(k,N\right)^*$. However, we show that almost all $k$ in the range $\left\lbrack3\pi\sqrt{n}, 5\pi\sqrt{n}\right\rbrack$ satisfy
\(
\mathrm{Tr}~\mathcal{T}_n\left(k,N\right)^* =O_\epsilon \left(k^{\frac{1}{2}+\epsilon}\right);
\)
see Theorem~\ref{selbergmain1}.
\end{remark}

\end{theorem}
We give the proof of the above theorem in Section~\ref{rmweight}. The proof is based on the Petersson trace formula and the proof of Theorem~\ref{maint} that we give  in Section~\ref{peters}.  The main term of the above formula comes from the $J$-Bessel function in the transition range. Next,  we simplify the error term by using bounds on the $J$-Bessel function outside the transition range. For the remaining error terms,  we average over weights and  apply the Poisson summation formula and obtain a sum of the Kloosterman sums twisted by oscillatory integrals.  Theorem~\ref{noweight} subsequently  follows by using Weil's bound for the Kloosterman sums and by exploiting the cancellation coming from the  summation over the Bessel functions (Section~\ref{secwe}). There are some similarities between our method and the circle method, especially the version developed by Heath-Brown~\cite{Heath}.


\subsubsection{Variance of the trace}
If we consider the variance of the trace of the Hecke operator over $k\sim \sqrt{n}$, the largeness of the trace in Theorem \ref{noweight} is no longer present. To be precise, we have the following  results.
\begin{theorem}\label{selbergmain1}
Let $N>1$ be a squarefree integer. For any positive integer $n$, we have
\[
\sum_{\substack{ k\in 2\mathbb{Z}\\3\pi\sqrt{n}<k < 5\pi\sqrt{n}}} \left|\mathrm{Tr}~\mathcal{T}_n\left(k,N\right)^*- \frac{k-1}{12} \varphi \left(N\right)\frac{\delta\left(n,\square\right)}{\sqrt{n}}\right|^2 \ll_{N} n\left(\log n\right)^2\left(\log\log n\right)^4,
\]
where $\delta\left(n,\square\right)=1$ if $n$ is a square, and $0$ otherwise. Here $\varphi$ is Euler's totient function. In particular, almost all $k$ in the range $\left\lbrack3\pi\sqrt{n}, 5\pi\sqrt{n}\right\rbrack$ satisfy
\[
\mathrm{Tr}~\mathcal{T}_n\left(k,N\right)^* =O_\epsilon \left(k^{\frac{1}{2}+\epsilon}\right).
\]
\end{theorem}
We also prove a lower bound for the variance of the trace of the Hecke operator. To make a precise statement, let $\phi$ be a positive even rapidly decaying function whose Fourier transform $\hat{\phi}$ is supported in $\left\lbrack-\frac{1}{100},\frac{1}{100}\right\rbrack$.
\begin{theorem}\label{selbergmain2}
Let $N>1$ be a squarefree integer and let $n=p^m$, where $p$ is an odd prime.  There exists a sufficiently large fixed constant $A>0$ such that for any $K>A\sqrt{n}$, we have
\begin{equation}\label{selberg2}
\frac{1}{\sum_{k\in 2\mathbb{Z}} \phi\left(\frac{k-1}{K}\right) }\sum_{k>0, k\in 2\mathbb{Z}} \phi\left(\frac{k-1}{K}\right) \left|\mathrm{Tr}~\mathcal{T}_n\left(k,N\right)^*- \frac{k-1}{12} \varphi \left(N\right)\frac{\delta\left(n,\square\right)}{\sqrt{n}}\right|^2 \gg_N n^{\frac{1}{2}}.
\end{equation}
\end{theorem}
This immediately implies the following weaker version of Theorem~\ref{rmwe}.
\begin{corollary}\label{cor1}
Let $N>1$ be a fixed square-free integer and let $p$ be an odd prime. Then we have
\[
D\left(\mu_{k,N}^*,\mu_p\right) =\Omega \left(\frac{1}{k^{\frac{1}{2}}(\log k)^2}\right).
\]
\end{corollary}
\begin{remark}
Note that this generalizes \cite{Gamburd} to any square-free level $N>1$.
\end{remark}
Theorem \ref{selbergmain1} and \ref{selbergmain2} are consequences of the following asymptotic formula, which we derive from the Eichler--Selberg trace formula for $T\geq \sqrt{n}$ and $N>1$ (Lemma \ref{sufficient11}):
\begin{multline}\label{suff}
\sum_{k>0, k\in 2\mathbb{Z}} \phi\left(\frac{k-1}{T}\right) \left|\mathrm{Tr}~\mathcal{T}_n\left(k,N\right)^*- \frac{k-1}{12} \varphi \left(N\right)\frac{\delta\left(n,\square\right)}{\sqrt{n}}\right|^2\\
= 2\sum_{k\in 2\mathbb{Z}} \phi\left(\frac{k-1}{T}\right) \sum_{t^2<4n}|D_N\left(t,n\right)|^2 - \phi\left(\frac{1}{T}\right)\frac{\sigma_1\left(n\right)^2}{n}   + O\left(n^{\frac{1}{2}+\epsilon}\right).
\end{multline}
Here $D_N\left(t,n\right)$ is a weighted sum of class numbers:
\[
D_N\left(t,n\right) = \frac{i}{2\sqrt{4n-t^2}} \sum_{f} h_w \left(\frac{t^2-4n}{f^2}\right) \tilde{\mu}\left(t,f,n,N\right),
\]
with weights $|\tilde{\mu}\left(t,f,n,N\right)| = O_N\left(1\right)$ (for the precise definition, see Lemma \ref{selbergnew}).

The upper bound (Theorem \ref{selbergmain1}) then follows by applying a standard upper bound for the class numbers of imaginary quadratic fields.

Note that inputting the sharp lower bound for the class numbers of imaginary quadratic fields,
\[
h_w\left(-d\right) \gg_\epsilon d^{\frac{1}{2}-\epsilon},
\]
to \eqref{suff} is not sufficient to prove the lower bound in Theorem \ref{selbergmain2}. Therefore we relate the problem of estimating the sparse sum of sums of class numbers
\[
\sum_{t^2<4n}|D_N\left(t,n\right)|^2
\]
to the problem of counting integral lattice points on $3$-spheres, under certain congruence conditions on the coordinates. This can be done by following the circle method developed by Kloosterman \cite{Kloosterman}, and we are able to show that
\[
\sum_{t^2<4n}|D_N\left(t,n\right)|^2 \gg_N \sqrt{n},
\]
under the assumption that $n$ is odd (Theorem \ref{arith}). Now if $n=p^m$ for a fixed odd prime $p$, and if $T > A \sqrt{n}$ for some large $A$, we see that
\[
2\sum_{k\in 2\mathbb{Z}} \phi\left(\frac{k-1}{T}\right) \sum_{t^2<4n}|D_N\left(t,n\right)|^2
\]
is larger than $\phi\left(\frac{1}{T}\right)\frac{\sigma_1\left(n\right)^2}{n} \gg n$, from which Theorem \ref{selbergmain2} follows. These steps are carried out in Section \ref{selberg}.

\subsubsection{Large discrepancy for the measure with harmonic weights}

Next, we give our results on the error term of the Weyl law associated to the measures $\nu_{k,N}^*$ as $k \to \infty$.

\begin{theorem}\label{weylpeter} Let $N\geq 1$ be a fixed square-free integer.
There exists an infinite sequence of weights $\{ k_n\}$ with $k_n\to \infty$ such that
\begin{equation}
D\left(\nu_{k_n,N}^*, \mu_{\infty}\right) \gg \frac{1}{k_n^{\frac{1}{3}}(\log k_n)^2}.
\end{equation}
\end{theorem}
\begin{remark} The above exceptional sequence of weights  is  very explicit  and is given by
$k_n=\lfloor4\pi p^n \rfloor$.
Based on heuristics stemming from arithmetic quantum chaos, numerical evidence \cite[Figure 5 and Figure 6]{Gamburd}, and the random model described in the introduction  for the eigenvalues of the Hecke operator, it is expected that
\begin{equation}\label{upb}
D\left(\mu_{k,N}^*,\mu_p\right)=O_{\epsilon,N}\left(k^{-\frac{1}{2}+\epsilon}\right) \text{ and } D\left(\nu_{k,N}^*,\mu_{\infty}\right)=O_{\epsilon,N}\left(k^{-\frac{1}{2}+\epsilon}\right)
\end{equation}
for a density 1 set of $k$. In this context, the exponent $\frac{1}{3}$ in Theorem~\ref{weylpeter} (and Theorem~\ref{rmwe}) shows that one can not achieve \eqref{upb} for every even  weight $k$.
\end{remark}

Theorem \ref{weylpeter} is an immediate consequence of an explicit asymptotic formula for the Petersson trace formula. More precisely, let $B_{k,N}^*$ be the orthonormal basis of $S_{k}(N)^*$ consists of holomorphic Hecke cusp forms, and let
\[
 \Delta_{k,N}^*\left(m,n\right):= \frac{\Gamma(k-1)}{(4\pi)^{k-1}} \sum_{f\in B_{k,N}^*} \rho_f\left(m\right) \overline{\rho_f\left(n\right)}.
\]

\begin{theorem}\label{maint} Let $N\geq 1$ be a fixed square-free integer.
Assume that $|4\pi \sqrt{mn}-k|<2k^{\frac{1}{3}}$ and $\gcd\left(mn,N\right)=1$.  Then
$$
\Delta_{k,N}^*\left(m,n\right)=\frac{\varphi\left(N\right)}{N}\delta\left(m,n\right)+2\pi i^{-k} \frac{\mu\left(N\right)}{N} \prod_{p|N} \left(1-\frac{1}{p^2}\right)J_{k-1}\left(4\pi  \sqrt{mn}\right)+O_N\left(k^{-\frac{1}{2}}\right),
$$
where $\delta\left(m,n\right)=1$ if $m=n$ and $\delta\left(m,n\right)=0$ otherwise.
\end{theorem}
\begin{remark}
Since $|4\pi \sqrt{mn}-k|<2k^{\frac{1}{3}}$, by the asymptotic behavior of the $J$-Bessel function in the transition range \eqref{trans}, we have $|J_{k-1}\left(4\pi  \sqrt{mn}\right)| \gg \frac{1}{k^{\frac{1}{3}}}$.  It follows that
\[
2\pi i^{-k} \frac{\mu\left(N\right)}{N} \prod_{p|N} \left(1-\frac{1}{p^2}\right)J_{k-1}\left(4\pi  \sqrt{mn}\right)
\]
is the main term, and
$$
|\Delta_{k,N}^*\left(m,n\right)-\delta\left(m,n\right)|\gg \frac{1}{k^{\frac{1}{3}}}.
$$
The above lower bound violates the naive expected square root cancelation in the sum of the normalized Fourier coefficients of the newforms in this range. More generally, one can generalize Theorem~\ref{maint}  if $\left|\frac{4\pi \sqrt{mn}}{q}-k\right|<2k^{\frac{1}{3}}$ for any fixed integer $q>0$.  In the appendix by Simon Marshall, the existence of this asymptotic trace formula is explained via the geometric side of the Petersson trace formula.
\end{remark}
We prove Theorem~\ref{maint} in Section~\ref{peters} by  applying the Petersson trace formula and  partitioning the geometric side of this formula into three parts according to the various behavior of the $J$-Bessel function in different ranges. This partition is explained in the appendix according to the incidence of the associated pairs of horocycles.

Theorem~\ref{noweight} follows from Theorem~\ref{maint} upon averaging over the parameters $m$ and $k$. In fact  we expect that a stronger version of Theorem~\ref{noweight} to be true, namely
 $$
  \mathrm{Tr}~\mathcal{T}_n\left(k,N\right)^*=  (-1)^{k/2} \frac{\mu\left(N\right)k}{\pi}\frac{\sigma\left(n\right)}{n} J_k\left(4\pi\sqrt{n}\right)  \left(1+o\left(1\right)\right),
 $$
where $k=4\pi\sqrt{n}+o\left(n^{\frac{1}{6}}\right)$. However, removing the harmonic weights in Kuznetsov's formula  by only averaging over $m$ in our context is equivalent to a very strong unproven bound for the $L$-functions, namely:
\begin{hyp}\label{twistedsym}
Let $n=O\left(k^2\right)$ and $N$ be a fixed square free integer.  Then
\begin{equation}
   \frac{\Gamma(k-1)}{(4\pi)^{k-1}}\sum_{f\in B_{k,N}^*} \left|\rho_f(1)\right|^2 \lambda_f\left(n\right)L\left(\frac{1}{2}+it,\mathrm{sym}^2f\right)  =O_{\delta}\left(k^{-\frac{1}{6}-\delta}\right),
\end{equation}
where $t=O\left(\log k^A\right)$ for some $A>0$ and $\delta>0$.
\end{hyp}
We overcome  this problem by averaging  over $k$ in a very short interval.




\section{Large discrepancy}\label{peters}
In this section, we deal with the lower bounds for the discrepancies
\[
D\left(\mu_{k_,N}^*, \mu_p\right),
\]
and
\[
D\left(\nu_{k_,N}^*, \mu_\infty\right).
\]
The main technical input for the lower bounds that we prove is an explicit asymptotic formula for the Petersson trace formula, Theorem \ref{maint}. Before we go into the details, we review some preliminary facts that are going to be used in the subsequent sections.
\subsection{Preliminary}
\subsubsection{$J$-Bessel function}\label{Bessel}
We first collect here estimates for the $J$-Bessel function from \cite{NIST:DLMF}.

When $\alpha \geq 0$ and $0<x \leq 1$, we have \cite[10.14.7]{NIST:DLMF}
\begin{equation}\label{large}
1\leq \frac{J_{\alpha}\left(\alpha x\right)}{x^{\alpha}J_{\alpha}\left(\alpha\right)}  \leq e^{\alpha\left(1-x\right)}.
\end{equation}
Note that $xe^{1-x}<1$ for $0<x<1,$ and  $0<J_{\alpha}\left(\alpha\right) \ll \frac{1}{\alpha^{\frac{1}{3}}}$ as $\alpha \to +\infty$; see~\eqref{trans}. Hence, \eqref{large} implies that $J_\alpha(\alpha x)$ is positive and exponentially small in $\alpha$ for any fixed $0<x <1$ as $\alpha \to +\infty$.

The transition range of the $J$-Bessel function $J_\alpha(y)$ is the range where $y$ is close to $\alpha$, i.e.,
\[
\alpha - c \alpha^{\frac{1}{3}} < y < \alpha + c \alpha^{\frac{1}{3}}
\]
is satisfied for some fixed constant $c>0$. In this range, we write $y=\alpha+a\alpha^{\frac{1}{3}}$, and the $J$-Bessel function has an asymptotic  in terms of the Airy function $Ai$~\cite[Theorem 1]{krasikov}
\begin{equation}\label{tranairy}
J_{\alpha}\left(\alpha+a\alpha^{\frac{1}{3}}\right)=  \frac{2^{\frac{1}{3}}}{\alpha^{\frac{1}{3}}}Ai\left(-2^{\frac{1}{3}}a\right) \left(1+O\left(\frac{1}{\alpha^{\frac{2}{3}}}\right)\right),
\end{equation}
where $a=O(1)$ and $\alpha\to \infty$. (See \cite[10.19.8]{NIST:DLMF} for the full asymptotic expansion of $J_\alpha(x)$ in this range.) Note that all zeros of Airy function $Ai(x)$ are negative, and the first zero is approximately $-2.33811\ldots $ \cite{MR0167642}. This implies that for $|a| <1$, we have
\begin{equation}\label{trans}
\frac{1}{\alpha^{\frac{1}{3}}} \ll J_{\alpha}\left(\alpha+a\alpha^{\frac{1}{3}}\right) \ll \frac{1}{\alpha^{\frac{1}{3}}}.
\end{equation}
We also have the following uniform upper bound for $ \frac{1}{2} \leq x < 1$,
\begin{equation}\label{uj0}
|J_\alpha \left(\alpha x\right)| \ll \frac{1}{\left(1-x^2\right)^{\frac{1}{4}}\alpha^{\frac{1}{2}}},
\end{equation}
and for $x\geq 1$,
\begin{equation}\label{uj}
|J_\alpha \left(\alpha x\right)| \ll \frac{1}{\left(x^2-1\right)^{\frac{1}{4}}\alpha^{\frac{1}{2}}}.
\end{equation}
If we combine \eqref{large}, \eqref{trans}, \eqref{uj0}, and \eqref{uj}, we have
\begin{equation}\label{wellknown}
|J_\alpha(y)| \ll \frac{1}{\alpha^\frac{1}{3}}.
\end{equation}

\subsubsection{Kloosterman sum}
For integers $m$, $n$, and $c\geq 1$, the Kloosterman sum $S(m,n;c)$ 
is defined by
\[
S(m,n;c) = \sum_{\substack{x \pmod{c}\\ \gcd(x,c)=1}} e\left(\frac{mx+nx^*}{c}\right),
\]
where $e(x) = \exp(2\pi i x)$, and $x^*$ is the multiplicative inverse of $x$ modulo $c$. We frequently use the following bound for the Kloosterman sum
\begin{equation}\label{weil}
|S\left(m,n;c\right)| \leq \sigma\left(c\right) \sqrt{\gcd\left(m,n,c\right)} \sqrt{c},
\end{equation}
which is often referred as Weil's bound.

\subsubsection{Petersson trace formula}
The Petersson trace formula \cite{Hans} is given by
\begin{multline}\label{peter}
\Delta_{k,N}\left(m,n\right):=\frac{\Gamma\left(k-1\right)}{\left(4\pi\right)^{k-1}}\sum_{f\in B_{k,N}}\rho_{f}\left(m\right) \overline{\rho_{f}\left(n\right)}\\
=\delta\left(m,n\right)+2\pi i^{-k}\sum_{c\equiv 0 \tpmod{N} }
\frac{S\left(m,n;c\right)}{c}J_{k-1}\left(\frac{4\pi \sqrt{mn}}{c}\right).
\end{multline}
For $M|N$, each newform $f$ of level $M$ gives rise to $\sigma\left(\frac{N}{M}\right)$ old forms in $S_k\left(N\right)$ \cite{Atkin}. By choosing a special orthonormal basis of $S_k\left(N\right)$, one may deduce the Petersson trace formula only for the newforms of squarefree level $N$ \cite[Proposition 2.9]{Luo}
\begin{multline}\label{newpeter}
\Delta_{k,N}^*\left(m,n\right):= \frac{\Gamma\left(k-1\right)}{\left(4\pi\right)^{k-1}}\sum_{f\in B_{k,N}^*}\rho_{f}\left(m\right) \overline{\rho_{f}\left(n\right)}\\
= \sum_{LM=N} \frac{\mu\left(L\right)}{L}\sum_{l|L^{\infty}} \frac{1}{l}\Delta_{k,M}\left(ml^2,n\right),
\end{multline}
where $\gcd(mn,N)=1.$
Henceforth, we assume that  $\gcd\left(mn,N\right)=1$  and
\begin{equation}\label{kn}
|4\pi\sqrt{mn}-k|<2k^{\frac{1}{3}}.
\end{equation}

\subsection{ Proof of Theorem~\ref{maint}.}
\begin{proof}
We apply the identity \eqref{newpeter} and obtain
\[
\Delta_{k,N}^*(m,n)= \sum_{LM=N} \frac{\mu(L)}{L}\sum_{l|L^{\infty}} \frac{1}{l}\Delta_{k,M}(ml^2,n).
\]
First, we analyze the contribution from $\delta\left(ml^2,n\right)$ which occurs when we apply the Petersson trace formula~\eqref{peter} to \eqref{newpeter}. Since $l|N^{\infty}$ and $\gcd\left(N,mn\right)=1$, the condition $ml^2=n$ can only be met if $l=1$ and $m=n$. By summing over $l$, we obtain
\[
 \sum_{LM=N} \frac{\mu\left(L\right)}{L}\sum_{l|L^{\infty}} \frac{1}{l}\delta\left(ml^2,n\right)= \sum_{LM=N} \frac{\mu\left(L\right)}{L}\delta\left(m,n\right)= \frac{\varphi\left(N\right)}{N}\delta\left(m,n\right).
\]
Therefore
\begin{equation*}
\Delta_{k,N}^*\left(m,n\right)=\frac{\varphi\left(N\right)}{N}\delta\left(m,n\right) + S_1+S_2,
\end{equation*}
where
\begin{equation}\label{S1}
S_1:=2\pi i^{-k} \frac{\mu\left(N\right)}{N} \prod_{p|N} \left(1-\frac{1}{p^2}\right) J_{k-1}\left(4\pi  \sqrt{mn}\right),
\end{equation}
\begin{equation*}
 S_2:=2\pi i^{-k} \sum_{LM=N} \frac{\mu\left(L\right)}{L}\sum_{l|L^{\infty}} \frac{1}{l}\sum_{\substack{ c\equiv 0 \pmod{M}\\ c\neq l}} \frac{S\left(ml^2,n;c\right)}{c}J_{k-1}\left(\frac{4\pi l \sqrt{mn}}{c}\right).
\end{equation*}
We have broken up the sum over $c\equiv 0 \pmod{N}$ into the terms $S_1$, for
which $c=l$, and $S_2$, for which $c\neq l$. The term $S_1$ has been simplified using the fact that the condition $c=l$ restricts the summation over $LM=N$ to $L=N$ and $M = 1,$ since $M|c, l|L,$ and $\gcd\left(L,M\right) = 1$, together with the fact that $S\left(ml^2; n; l\right) = S\left(0; n; l\right)=\mu\left(l\right)$.

By the estimate \eqref{trans} and the assumption   \eqref{kn}, we have $|S_1| \gg_N \frac{1}{k^{\frac{1}{3}}}$. Next, we give an upper bound for $S_2$. For $\delta>0$ to be chosen, let
$$
S_{2,\delta}:= 2\pi i^{-k} \sum_{LM=N} \frac{\mu\left(L\right)}{L}
\sum_{\substack{l>k^{\delta}\\l|L^{\infty}}}
 \frac{1}{l}\sum_{ \substack{ c\equiv 0 \pmod{M}\\ c\neq l}} \frac{S\left(ml^2,n;c\right)}{c}J_{k-1}\left(\frac{4\pi l \sqrt{mn}}{c}\right).
$$
For $k^{\delta}>N$, it follows from \eqref{newpeter} and \eqref{S1} that
$$
S_{2,\delta}=\sum_{LM=N} \frac{\mu\left(L\right)}{L}\sum_{\substack{l>k^{\delta}\\l|L^{\infty}}}\frac{1}{l} \left( \Delta_{k,M}\left(ml^2,n\right)-\delta\left(ml^2,n\right)\right).
$$
By \cite[Corollary 2.2]{Luo}, we have
$$
\Delta_{k,M}\left(ml^2,n\right)-\delta\left(ml^2,n\right)=O_{N,\epsilon}\left(\frac{\left(mn\right)^{\frac{1}{4}+\epsilon} l^{\frac{1}{2}+\epsilon}}{k^{\frac{5}{6}}} \right).
$$
Therefore
$$
|S_{2,\delta}| \ll_{N,\epsilon} \sum_{\substack{l>k^{\delta}\\l|N^{\infty}}} \frac{1}{l}\frac{\left(mn\right)^{\frac{1}{4}+\epsilon} l^{\frac{1}{2}+\epsilon}}{k^{\frac{5}{6}}}.
$$
By \eqref{kn}, we have
\begin{equation}\label{S2d}
|S_{2,\delta}| \ll_{N,\epsilon} {k^{-\frac{1}{3}+\epsilon}}  \sum_{\substack{l>k^{\delta}\\l|N^{\infty}}}  l^{-\frac{1}{2}+\epsilon} =O_{N,\epsilon}\left({k^{-\frac{1}{3}-\frac{\delta}{2}+2\epsilon}}\right).
\end{equation}
Finally, we give an upper bound for $S\left(\delta\right):=S_2-S_{2,\delta}$. We split $S\left(\delta\right)$ into three parts, each of which has a restriction on the sum over $c\equiv 0 \pmod{M}$.  We write $S_i\left(\delta\right)$ for the sum $S\left(\delta\right)$ subjected to the $i$th condition listed below.
\begin{enumerate}
\item $2l<c$
\item $l<c<2l$
\item $c<l$
\end{enumerate}
By  \eqref{large}, \eqref{trans} and \eqref{weil}, we first have
 \begin{align*}
 |S_1\left(\delta\right)| &\ll \left| \sum_{LM=N} \frac{\mu\left(L\right)}{L}\sum_{\substack{l<k^{\delta}\\l|L^{\infty}}} \frac{1}{l}\sum_{\substack{c \equiv 0 \pmod{M}\\c>2l}} \frac{S\left(ml^2,n;c\right)}{c}J_{k-1}\left(\frac{4\pi l \sqrt{mn}}{c}\right) \right|\\
 &\ll_N  \sum_{\substack{l|N^{\infty}\\ l<k^{\delta}}} \frac{1}{l}\sum_{c>2l } \left|\frac{S\left(ml^2,n;c\right)}{c} J_{k-1}\left(\frac{4\pi l \sqrt{mn}}{c}\right)\right|\\
  &\ll_N  \sum_{\substack{l|N^{\infty}\\ l<k^{\delta}}} \frac{1}{l}\sum_{c>2l }  \frac{e^{k\left(1-\frac{l}{c}+\log\left(\frac{l}{c}\right)\right)}}{k^{\frac{1}{3}}}\\
 & \ll_N \sum_{\substack{l|N^{\infty}\\ l<k^{\delta}}} \frac{ e^{k\left(1-\frac{1}{2}-\log\left(2\right)\right)} }{k^{\frac{1}{3}}}\ll_{N,\delta}  e^{-\left(0.19\right) k}. \numberthis \label{tail}
 \end{align*}
Next, we give an upper bound for $S_2\left(\delta\right)$ and  $S_3\left(\delta\right)$. Assume that $l<c<2l<2k^{\delta}$. By the  inequality \eqref{uj0}, \eqref{kn} and \eqref{weil}
\begin{align*}
|S_2\left(\delta\right)| &\ll \left| \sum_{LM=N} \frac{\mu\left(L\right)}{L}\sum_{\substack{l<k^{\delta}\\l|L^{\infty}}} \frac{1}{l}\sum_{\substack{ c\equiv 0 \pmod{M}\\ c<2l} } \frac{S\left(ml^2,n;c\right)}{c}J_{k-1}\left(\frac{4\pi l \sqrt{mn}}{c}\right) \right|
\\
& \ll_{N,\epsilon} \sum_{\substack{l|N^{\infty}\\ l<k^{\delta}}} \frac{1}{l}\sum_{l<c<2l } \sqrt{\gcd\left(m,n,c\right)} c^{-\frac{1}{2}+\epsilon} \left| J_{k-1}\left(\frac{4\pi l \sqrt{mn}}{c}\right) \right|
\\
& \ll_{N,\epsilon} \sum_{\substack{l|N^{\infty}\\ l<k^{\delta}}} \frac{1}{l}\sum_{l<c<2l } \sqrt{\gcd\left(m,n,c\right)} c^{-\frac{1}{2}+\epsilon} k^{-\frac{1}{2}}  \frac{1}{\left(1-\frac{l^2}{c^2}\right)^{\frac{1}{4}}}
\\
& \ll_{N,\epsilon}   k^{-\frac{1}{2}} \sum_{\substack{l|N^{\infty}\\ l<k^{\delta}}} l^{-\frac{5}{4}+\epsilon}\sum_{l<c<2l } \frac{\sqrt{\gcd\left(m,n,c\right)} }{\left(c-l\right)^{\frac{1}{4}}}
\\
& \ll_{N,\epsilon} k^{-\frac{1}{2}}  \sum_{\substack{l|N^{\infty}\\ l<k^{\delta}}} l^{-\frac{1}{2}+\epsilon} \ll_N k^{-\frac{1}{2}}, \numberthis \label{eq11}
\end{align*}
where we let $\epsilon=1/100$ in the last estimate. Finally,  assume that $c<l<k^{\delta}.$ Then by \eqref{uj} and \eqref{weil}
\begin{align*}
|S_3\left(\delta\right)| &\ll \left| \sum_{LM=N} \frac{\mu\left(L\right)}{L}\sum_{\substack{l<k^{\delta}\\l|L^{\infty}}} \frac{1}{l}\sum_{\substack{ c \equiv 0 \pmod{M}\\ c<l }} \frac{S\left(ml^2,n;c\right)}{c}J_{k-1}\left(\frac{4\pi l \sqrt{mn}}{c}\right) \right|
\\
& \ll_{N,\epsilon} \sum_{\substack{l|N^{\infty}\\ l<k^{\delta}}} \frac{1}{l}\sum_{c<l } \sqrt{\gcd\left(m,n,c\right)} c^{-\frac{1}{2}+\epsilon} \left| J_{k-1}\left(\frac{4\pi l \sqrt{mn}}{c}\right) \right|
\\
& \ll_{N,\epsilon} \sum_{\substack{l|N^{\infty}\\ l<k^{\delta}}} \frac{1}{l}\sum_{c<l } \sqrt{\gcd\left(m,n,c\right)} c^{-\frac{1}{2}+\epsilon} k^{-\frac{1}{2}}  \frac{1}{\left(\frac{l^2}{c^2}-1\right)^{\frac{1}{4}}}
\\
& \ll_{N,\epsilon}   k^{-\frac{1}{2}} \sum_{\substack{l|N^{\infty}\\ l<k^{\delta}}} l^{-\frac{5}{4}}\sum_{c<l } \sqrt{\gcd\left(m,n,c\right)} c^{-\frac{1}{2}+\epsilon} \frac{c^{\frac{1}{2}}}{\left(l-c\right)^{\frac{1}{4}}}
\\
& \ll_{N,\epsilon} k^{-\frac{1}{2}}  \sum_{\substack{l|N^{\infty}\\ l<k^{\delta}}} l^{-\frac{1}{2}+\epsilon} \ll_N k^{-\frac{1}{2}}, \numberthis \label{eq22}
\end{align*}
again, where we let  $\epsilon=1/100$ in the last estimate.

Now let $\delta=1$ and combine \eqref{S2d}, \eqref{tail},  \eqref{eq11} and \eqref{eq22} to obtain
\begin{multline*}
\Delta_{k,N}^*\left(m,n\right)=
\frac{\varphi\left(N\right)}{N}\delta\left(m,n\right)+2\pi i^{-k} \frac{\mu\left(N\right)}{N} \prod_{p|N} \left(1-\frac{1}{p^2}\right)J_{k-1}\left(4\pi  \sqrt{mn}\right)+O_N\left(k^{-\frac{1}{2}}\right).\qedhere
\end{multline*}
\end{proof}
\subsection{ Proof of Theorem~\ref{weylpeter}.}
\begin{proof}
Recall that
\[
 \nu_{k,N}^*:= \frac{\left(4\pi\right)^{k-1}}{\Gamma\left(k-1\right)}\sum_{f\in B_{k,N}^*} |\rho_f(1)|^2\delta_{\lambda_f\left(p\right)}.
\]
Since $|\lambda_f\left(p\right)| \leq 2$, we can write   $\lambda_f\left(p\right)=2\cos\left(\theta_f\left(p\right)\right)$ for a unique  $0 \leq \theta_f\left(p\right) \leq \pi$. Let $U_n\left(\cos \theta\right)=\frac{\sin\left(n+1\right)\theta}{\sin\theta}$ for $n\geq 0$ be the $n$th Chebyshev polynomial of the second
kind. Recall from \cite[Lemma 3]{Conrey} that $\lambda_{f}\left(p^n\right)=U_n\left(\frac{\lambda_f\left(p\right)}{2}\right)$.  In order to give a lower bound for the discrepancy between  $\nu_{k_n,N}^*$ and $\mu_{\infty}$  for $k_n:= \lfloor 4\pi \sqrt{p^n} \rfloor$, we compute the difference between  the expected value of $U_n\left(\frac{x}{2}\right)$ with respect to these measures. Note that $\{U_n(\frac{x}{2})\}_{n=0}^\infty$ is an orthogonal set of polynomials with respect to $\mu_{\infty}$ \cite[7.343.2]{GR}. Hence for $n \geq 1$,
\[
\int_{-2}^{2} U_n\left(\frac{x}{2}\right) d\mu_{\infty}\left(x\right)=0.
\]
On the other hand, by Theorem~\ref{maint}, since $|k_n-4\pi\sqrt{p^n}|<1$  we have
\begin{multline*}
 \int_{-2}^{2} U_n\left(\frac{x}{2}\right)  d\nu_{k,N}^*\left(x\right)=\Delta_{k_n,N}^*\left(1,p^n\right)\\
 =2\pi i^{-k}\frac{\mu\left(N\right)}{N} \prod_{p|N}\left(1-\frac{1}{p^2}\right)J_{k_n-1}\left(4\pi  \sqrt{p^n}\right) +O_N\left(k^{-\frac{1}{2}}\right).
\end{multline*}
As pointed out in Remark~\ref{tranrem}, since $|k_n-4\pi\sqrt{p^n}|<1$ then by \eqref{trans}, we have
$$
 \int_{-2}^{2} U_n\left(\frac{x}{2}\right)  d\nu_{k,N}^*\left(x\right)\gg_N k_n^{-\frac{1}{3}}.
$$
By integration by parts and the upper bound $|U_n^{\prime}\left(\frac{x}{2}\right)| \ll n^2$, it follows that
 \begin{equation}\label{Erdus}
D\left(\nu_{k_n,N}^*, \mu_{\infty}\right) \gg \frac{1}{n^2k_n^{\frac{1}{3}}} .
\end{equation}
Since $k_n=\lfloor 4\pi\sqrt{p^n} \rfloor$, we conclude that
\[
D\left(\nu_{k_n,N}^*, \mu_{\infty}\right) \gg  \frac{1}{k_n^{\frac{1}{3}}(\log k_n)^2}. \qedhere
\]
\end{proof}

\section{Removing the weights}\label{rmweight}
In this section we give the proof of Theorem~\ref{noweight}, from which Theorem~\ref{rmwe} follows as a corollary. Note that the trace of the Hecke operator $\mathcal{T}_n\left(N,k\right)^*$ is obtained by  removing the arithmetic weights $\left|\rho_f(1)\right|^2$ from the Petersson trace formula~\eqref{newpeter} with $m=1$. The usual trick  for removing these weights is to average the Petersson trace formula~\eqref{newpeter} smoothly over $m^2$ where $\gcd\left(m,N\right)=1$. Unfortunately, once summed over $m^2$, it is difficult to prove a bound for the contribution coming from $S_2\left(\delta\right)$ and $S_3\left(\delta\right)$ that is smaller than the contribution from the main term. We therefore sum the trace formula as $k$ varies inside a short interval of size  $\sim k^{\delta}$ for some $\frac{1}{5}<\delta <\frac{1}{3}$ and exploit the cancellation coming from the summation of $J$-Bessel function over the order $k$ (Lemma~\ref{avk}). We note here that $\delta<\frac{1}{5}$ is not large enough to bound the error term and  $\delta>\frac{1}{3}$ makes the main term smaller than the error term! Theorem~\ref{noweight} then follows from  Weil's bound for the Kloosterman sum and Lemma~\ref{avk}.
\subsection{Averaging over the weight}\label{secwe}
Recall that $\psi$ is a positive smooth function supported in $\left\lbrack-1,1\right\rbrack$ and $\int_{-1}^{1}\psi\left(t\right)dt=1$.   Let $K>0$ be a positive real number.
 \begin{lemma}\label{avk}
Fix $0<\delta<\frac{1}{3}$ and $\eta >1-3\delta>0$. Let $x>0$. If $|x-K| > K^{\eta+\delta}$,  then
\begin{equation}\label{sta1}
\sum_{l\equiv 1 \mod 2}\psi\left(\frac{l-K}{K^{\delta}}\right) J_l\left(x\right) \ll_{A,\psi,\eta,\delta}  K^{-A}.
 \end{equation}
If $|x-K| < K^{\eta+\delta}$, we have
 \begin{equation}\label{sta2}
\frac{1}{K^{\delta}}\sum_{l\equiv 1 \mod 2}\psi\left(\frac{l-K}{K^{\delta}}\right) J_l\left(x\right) \ll   K^{-\frac{1}{3}}.
\end{equation}
Moreover, if $x=K+o\left(K^{\frac{1}{3}}\right)$ then
\begin{equation}\label{sta3}
\frac{1}{K^{\delta}} \sum_{l\equiv 1 \mod 2}\psi\left(\frac{l-K}{K^{\delta}}\right) J_l\left(x\right) = J_K\left(x\right)\left(\frac{1}{2}+o_\psi(1)\right) \gg_{\psi} K^{-\frac{1}{3}}.
\end{equation}
 \end{lemma}
 \begin{proof}
As done in \cite[\S5.5]{Iwaniecb}, we use the integral representation of the $J$-Bessel function
$$
J_l\left(x\right)=\int_{-\frac{1}{2}}^{\frac{1}{2}} e^{-2\pi i lt}e^{-ix\sin2\pi t}dt,
$$
from \cite[8.411.1]{GR}. By the  Poisson summation formula, it follows that
\begin{equation*}
\begin{split}
\sum_{l\equiv 1 \mod 2}\psi\left(\frac{l-K}{K^{\delta}}\right)J_l\left(x\right)=\int_{-\infty}^{\infty}\hat{\psi}\left(u\right) e^{-2\pi iuK^{1-\delta}} \left( e^{-ix\sin\left(\frac{2\pi u}{K^{\delta}}\right)}-e^{ix\sin\left(\frac{2\pi u}{K^{\delta}}\right)}  \right) du.
\end{split}
\end{equation*}
Because $\hat{\psi}(u) \ll_A |u|^{-A}$, we may assume that the integral is taken over $|u|<K^\kappa$ with some $\kappa$ that satisfies
\[
0<\kappa<\min\left\{\delta,~ \frac{1}{2}(\eta-(1-3\delta)) \right\}.
\]
Since the remaining portion of the integral contributes a negligible
amount.
Let
\[
f_\pm (u) = -uK^{1-\delta}\pm \frac{x}{2\pi} \sin\left(\frac{2\pi u}{K^{\delta}}\right),
\]
and then we have
\[
f_\pm(u)' = -K^{1-\delta}\pm \frac{x}{K^{\delta}} \cos\left(\frac{2\pi u}{K^{\delta}}\right) =   \frac{-K\pm x}{K^{\delta}} + O \left(\frac{xu^2}{K^{3\delta}}\right).
\]
Now assume that $x>0$ and that $|K-x|> K^{\delta+\eta}$. If $x>2K$, then
\[
\left|f_\pm(u)'\right| \gg xK^{-\delta} + O(x K^{-\delta-2(\delta-\eta)}) \gg K^{1-\delta}.
\]
If $0<x<2K$, then
\[
\left|f_\pm(u)'\right| \gg K^{\eta} + O( K^{1+2\kappa - 3\delta}),
\]
and because $1+2\kappa-3\delta < 1-3\delta +(\eta-(1-3\delta)) =\eta$, we have
\[
\left|f_\pm(u)'\right| \gg K^{\eta}.
\]

Therefore, by repeated integration by parts, we have
$$
\int_{-K^{\kappa}}^{K^{\kappa}} \hat{\psi}\left(u\right) e^{-2\pi iuK^{1-\delta}} \left( e^{-ix\sin\left(\frac{2\pi u}{K^{\delta}}\right)}-e^{ix\sin\left(\frac{2\pi u}{K^{\delta}}\right)}  \right) du \ll_{A,\psi,\eta,\delta} K^{-A},
$$
for any $A>0$. This completes the proof of \eqref{sta1}.

The inequality~\eqref{sta2} follows from the upper bound \eqref{wellknown}, and the fact that $\psi$ is supported in $[-1,1]$. Finally, \eqref{sta3} follows from  the asymptotic of the $J$-Bessel function in the transition range. More precisely, recall \eqref{tranairy} and ~\eqref{trans}
\[
J_{\alpha}\left(\alpha+a\alpha^{\frac{1}{3}}\right)=  \frac{2^{\frac{1}{3}}}{\alpha^{\frac{1}{3}}}Ai\left(-2^{\frac{1}{3}}a\right) \left(1+O_{\delta}\left(\frac{1}{\alpha^{\frac{2}{3}}}\right)\right)\gg  \frac{1}{\alpha^{\frac{1}{3}}},\]
where $|a|<2.$ Hence, for $x=K+o\left(K^{\frac{1}{3}}\right)$ and $0<\delta<\frac{1}{3}$
\[
\frac{1}{K^{\delta}} \sum_{l\equiv 1 \mod 2}\psi\left(\frac{l-K}{K^{\delta}}\right) J_l\left(x\right) = J_K\left(x\right)\left(\frac{1}{2}+o_\psi(1)\right) \gg_{\psi} K^{-\frac{1}{3}}. \qedhere
\]
\end{proof}
First, we cite some identities from \cite{Luo} that we use in the proof. Let $f$ be a newform of $S_k\left(N\right)$ of level $M$. Then by  \cite[Lemma 2.5]{Luo}, we have
\begin{equation}\label{harm}
\rho_f(m)\overline{\rho_f(n)} = \frac{(4\pi)^{k-1}}{\Gamma(k)} \frac{12M \lambda_f(m)\lambda_f(n)}{\nu(N)\varphi(M) Z(1,f)},
\end{equation}
where $Z\left(s,f\right):=\sum_{n=1}^\infty \lambda_f\left(n^2\right)n^{-s}$. Note that $Z\left(s,f\right)$ is related to $L\left(s,\mathrm{sym}^2 f\right)$ by
$$
L\left(s,\mathrm{sym}^2 f\right)=\frac{\zeta\left(2s\right)}{\zeta_N\left(2s\right)}Z\left(s,f\right),
$$
where $\zeta_N\left(2s\right)=\prod_{p|N}\left(1-p^{-2s}\right)^{-1}$ \cite[(3.14)]{Luo}. Let
\[
Z^N\left(s,f\right):=\sum_{\substack{m=1\\ \gcd\left(m,N\right)=1}}^\infty \frac{\lambda_f\left(m^2\right)}{m^s}.
\]
Then by  \cite[(3.16)]{Luo},
\begin{equation}\label{Z^N} Z^N\left(s,f\right) = L\left(s,\mathrm{sym}^2 f\right)   \frac{\zeta_N\left(2s\right)}{\zeta\left(2s\right)\zeta_N\left(s+1\right)}.
\end{equation} By the celebrated result of Shimura~\cite{Shimura},  $L\left(s,\mathrm{sym}^2 f\right) $ is an entire function. Hence  $Z^N\left(s,f\right)$ is holomorphic  for $\Re\left(s\right)>\frac{1}{2}$ and has a meromorphic continuation to the complex plane.   Let $w\left(x\right)=\exp\left(-x\right)$. Note that the Mellin transform of $w$ is the Gamma function
\[
\hat{w}\left(s\right):=\int_{0}^{\infty} x^{s-1}w\left(x\right)dx=\Gamma\left(s\right).
\]

\subsection{Proof of Theorem~\ref{noweight}}
\begin{proof}
 Assume that $k\in [K-K^{\delta},K+K^{\delta}]$ where $\delta<\frac{1}{3}$.  By the Petersson formula \eqref{newpeter},
 \begin{equation}\label{rw}
\frac{\Gamma(k-1)}{(4\pi)^{k-1}}\sum_{f\in B_{k,N}^*}\rho_f\left(m^2\right) \overline{\rho_{f}\left(n\right)}= \sum_{LM=N} \frac{\mu\left(L\right)}{L}\sum_{l|L^{\infty}} \frac{1}{l}\Delta_{k,M}\left(m^2l^2,n\right).
\end{equation}
Let $T=K^{\alpha}$ for some fixed $0<\alpha<1$ that we choose at the end of the proof.  We average the left-hand side of the above by the smooth function $\frac{1}{x}w\left(\frac{x}{T}\right)$ and use \eqref{harm} to obtain
\begin{multline}\label{rhs}
\frac{\Gamma(k-1)}{(4\pi)^{k-1}}\sum_{\substack{m\geq 1\\ \gcd\left(m,N\right)=1}}\frac{1}{m}w\left(\frac{m}{T}\right)\sum_{f\in B_{k,N}^* }\rho_f\left(m^2\right) \overline{\rho_{f}\left(n\right)}\\
=\sum_{f\in B_{k,N}^* }\sum_{\substack{m\geq 1\\ \gcd\left(m,N\right)=1}}w\left(\frac{m}{T}\right) \frac{12 \lambda_f\left(n\right)\lambda_{f}\left(m^2\right)\zeta_N\left(2\right)}{m\left(k-1\right)NZ\left(1,f\right)} \\
= \frac{12 }{\left(k-1\right)N} \sum_{f\in B_{k,N}^* } \lambda_f\left(n\right)   \frac{\zeta_N\left(2\right)}{Z\left(1,f\right)}\sum_{\substack{m\geq 1\\ \gcd\left(m,N\right)=1}} w\left(\frac{m}{T}\right) \frac{\lambda_{f}\left(m^2\right)}{m }.
\end{multline}
By the  Mellin inversion formula, we have $w\left(\frac{x}{T}\right)=\frac{1}{2\pi i} \int_{2-i\infty}^{2+i\infty}\Gamma\left(s\right)T^s x^{-s}ds$ and this implies
\begin{equation*}
\sum_{\substack{m\geq 1\\ \gcd\left(m,N\right)=1}} w\left(\frac{m}{T}\right) \frac{\lambda_{f}\left(m^2\right)}{m}=\frac{1}{2\pi i}\int_{2-i\infty}^{2+i\infty} Z^N\left(s+1,f\right)T^s \Gamma\left(s\right)ds.
\end{equation*}
We shift the contour to the line $\Re\left(s\right)=-\frac{1}{2}$  and pick up the pole of $\Gamma\left(s\right)$ at $s=0$ with residue $Z^N\left(1,f\right)= \frac{Z\left(1,f\right)}{\zeta_N\left(2\right)}$, and hence
\begin{equation}\label{rezvan}
\sum_{\substack{m\geq 1\\ \gcd\left(m,N\right)=1}}w\left(\frac{m}{T}\right) \frac{\lambda_{f}\left(m^2\right)}{m}=\frac{Z\left(1,f\right)}{\zeta_N\left(2\right)}+\frac{1}{2\pi i}\int_{-\frac{1}{2}-i\infty}^{-\frac{1}{2}+i\infty}Z^N\left(s+1,f\right)T^s \Gamma\left(s\right)ds.
\end{equation}
By \eqref{Z^N},
\begin{multline*}
\frac{1}{2\pi i}\int_{-\frac{1}{2}-i\infty}^{-\frac{1}{2}+i\infty}Z^N\left(s+1,f\right)T^s \Gamma\left(s\right)ds\\
=  \frac{1}{2\pi i}\int_{-\infty}^{\infty} L\left(\frac{1}{2}+it,\mathrm{sym}^2 f\right)   \frac{\zeta_N\left(1+2it\right)}{ \zeta\left(1+2it\right) \zeta_N\left(\frac{3}{2}+it\right)}T^{-\frac{1}{2}+it} \Gamma\left(-\frac{1}{2}+it\right)dt.
 \end{multline*}
 First, we bound the portion of the integral for which $|t|>\left(\log k\right)^2$. By Stirling's formula \cite[5.11.9]{NIST:DLMF},
$$
\Gamma\left(-\frac{1}{2}+it\right)=O\left(\left(1+|t|\right)^{-1}e^{-\frac{\pi |t|}{2}}\right).
$$
By using the above bound,  the convexity bound \cite[(34)]{MR1826269}
\[
L\left(\frac{1}{2}+it,\mathrm{sym}^2 f\right) \ll_{\epsilon,N} k^{\frac{1}{2}+\epsilon}(|t|+1)^{\frac{3}{4}+\epsilon},
\]
the well-known bound $ \zeta\left(1+2it\right)^{-1}=O\left(\log (|t|+1)\right) $, the fact that $\zeta_N\left(2s\right) \zeta_N\left(s+1\right)^{-1}$ is bounded on $\Re\left(s\right)=\frac{1}{2}$ and  $|T^{-\frac{1}{2}+it}|\leq T^{-\frac{1}{2}}\leq k^{-\frac{\alpha}{2}},$ it follows that
\begin{multline*}
\left(\int_{-\infty}^{-(\log k)^2}+\int_{(\log k)^2}^{\infty}\right)     \frac{L\left(\frac{1}{2}+it,\mathrm{sym}^2 f\right)\zeta_N\left(1+2it\right) }{\zeta\left(1+2it\right)\zeta_N\left(\frac{3}{2}+it\right)}T^{-\frac{1}{2}+it} \Gamma\left(-\frac{1}{2}+it\right)dt\\
=O_A\left(k^{-A}\right)
\end{multline*}
for any $A>0$. By the above,  \eqref{rhs} and \eqref{rezvan}, we obtain
\begin{multline}\label{new}
\frac{\Gamma(k-1)}{(4\pi)^{k-1}}\sum_{\substack{m\geq 1\\ \gcd\left(m,N\right)=1}}\frac{1}{m}w\left(\frac{m}{T}\right)\sum_{f\in B_{k,N}^* }\rho_f\left(m^2\right) \overline{\rho_{f}\left(n\right)}=  \frac{12 }{\left(k-1\right)N}  \mathrm{Tr}~\mathcal{T}_n\left(N,k\right)^*
\\
+\frac{1}{2\pi i}\int_{-(\log k)^2}^{(\log k)^2} \left( \sum_{f\in B_{k,N}^* }\frac{12 \zeta_N(2)}{(k-1)NZ(1,f)}  \lambda_f\left(n\right)  L\left(\frac{1}{2}+it,\mathrm{sym}^2 f \right) \right)  \\
\times \frac{\zeta_N\left(1+2it\right)T^{-\frac{1}{2}+it} \Gamma\left(-\frac{1}{2}+it\right)}{\zeta\left(1+2it\right)\zeta_N\left(\frac{3}{2}+it\right)} dt+O\left(k^{-A}\right).
\end{multline}
From \eqref{deligne}, we have $ |\lambda_f\left(n\right)| \leq \sigma(n) \ll_\epsilon n^{\epsilon} \ll_\epsilon k^\epsilon$, and we also have $Z(1,f)^{-1} \ll_\epsilon k^\epsilon$ \cite[Theorem 0.2]{Hoffstein}. Therefore we have
\begin{multline*}
\frac{1}{2\pi i}\int_{-(\log k)^2}^{(\log k)^2} \left( \sum_{f\in B_{k,N}^* }\frac{12 \zeta_N(2)}{(k-1)NZ(1,f)}  \lambda_f\left(n\right)  L\left(\frac{1}{2}+it,\mathrm{sym}^2 f \right) \right)  \\
\times \frac{\zeta_N\left(1+2it\right)T^{-\frac{1}{2}+it} \Gamma\left(-\frac{1}{2}+it\right)}{\zeta\left(1+2it\right)\zeta_N\left(\frac{3}{2}+it\right)} dt\\
\ll_\epsilon k^{-1+\epsilon} T^{-\frac{1}{2}} \int_{-(\log k)^2}^{(\log k)^2}  \sum_{f\in B_{k,N}^* } \left| L\left(\frac{1}{2}+it,\mathrm{sym}^2 f \right)\right| dt.
\end{multline*}
To simplify the notation, we let
\[
\mathcal{M}_1(k) = \int_{-(\log k)^2}^{(\log k)^2}  \sum_{f\in B_{k,N}^* } \left| L\left(\frac{1}{2}+it,\mathrm{sym}^2 f \right)\right| dt.
\]
By the above  and \eqref{new}, we have
\begin{multline}\label{lhsp}
\frac{\Gamma(k-1)}{(4\pi)^{k-1}}\sum_{\substack{m\geq 1\\ \gcd\left(m,N\right)=1}}\frac{1}{m}w\left(\frac{m}{T}\right)\sum_{f\in B_{k,N}^* }\rho_f\left(m^2\right) \overline{\rho_{f}\left(n\right)}\\
=\frac{12 }{\left(k-1\right)N}  \mathrm{Tr}~\mathcal{T}_n\left(N,k\right)^*+O_\epsilon\left(T^{-\frac{1}{2}}k^{-1+\epsilon}\mathcal{M}_1(k)\right).
\end{multline}

Finally, we average  the right-hand side of \eqref{rw} with the same weights $\frac{1}{m}w\left(\frac{m}{T}\right)$. Let
$$
S:=\sum_{\substack{m\geq 1\\ \gcd\left(m,N\right)=1}} \frac{1}{m}w\left(\frac{m}{T}\right) \sum_{LM=N} \frac{\mu\left(L\right)}{L}\sum_{l|L^{\infty},} \frac{1}{l}\Delta_{k,M}\left(m^2l^2,n\right).
$$
We analyze the contribution of $\delta\left(m^2l^2,n\right)$ by applying the Petersson formula~\eqref{peter}. Since $l|N^{\infty}$ and $\gcd\left(N,mn\right)=1$, then the condition $m^2l^2=n$ can only be met when $l=1$ and $m^2=n$. Therefore,
\begin{multline*}
 \sum_{\substack{m\geq 1\\ \gcd\left(m,N\right)=1}} \frac{1}{m}w\left(\frac{m}{T}\right) \sum_{LM=N} \frac{\mu\left(L\right)}{L}\sum_{l|L^{\infty}} \frac{1}{l}\delta\left(m^2l^2,n\right)\\
 =  \frac{1}{\sqrt{n}}w\left(\frac{\sqrt{n}}{T}\right) \sum_{LM=N} \frac{\mu\left(L\right)}{L}\delta\left(n,\square\right)= \frac{\varphi\left(N\right)w\left(\frac{\sqrt{n}}{T}\right)}{N\sqrt{n}}\delta\left(n,\square\right).
\end{multline*}
Note that by our choice of $w$, if $T \ll n^{\frac{1}{2}-\epsilon}$, then
\begin{equation}\label{sqcont}
\frac{\varphi\left(N\right)w\left(\frac{\sqrt{n}}{T}\right)}{N\sqrt{n}}\delta\left(n,\square\right)=O_A\left(k^{-A}\right).
\end{equation}
Let
\[
S^T:=\sum_{\substack{m> T^{1+\epsilon}\\ \gcd\left(m,N\right)=1}}\frac{1}{m} w\left(\frac{m}{T}\right) \sum_{LM=N} \frac{\mu\left(L\right)}{L}\sum_{l|L^{\infty},} \frac{1}{l} \left(\Delta_{k,M}\left(m^2l^2,n\right)-\delta\left(m^2l^2,n\right)\right).
\]
By \cite[Corollary 2.2]{Luo}, we have
\[
\Delta_{k,M}\left(m^2l^2,n\right)-\delta\left(m^2l^2,n\right)=O_{N,\epsilon}\left(\frac{n^{\frac{1}{4}+\epsilon} \left(ml\right)^{\frac{1}{2}+\epsilon}}{k^{\frac{5}{6}}} \right).
\]
It follows from the above and the choice of $w$ and $T$ that
$S^T=O_A\left(k^{-A}\right)$. Hence
\begin{equation}\label{splits12}
S=S_1+S_2+O_A\left(k^{-A}\right),
\end{equation} where
\begin{multline*}
S_1:= 2\pi i^{-k}\sum_{LM=N} \frac{\mu\left(L\right)}{L}\sum_{l|L^{\infty}}  \sum_{\substack{m<T^{1+\epsilon}\\\gcd\left(m,N\right)=1}} \frac{1}{ml}w\left(\frac{m}{T}\right) \\
\times \sum_{c\equiv 0\pmod{M} }\delta(c,ml) \frac{S\left(m^2l^2,n;c\right)}{c}J_{k-1}\left(\frac{4\pi ml \sqrt{n}}{c}\right),
\end{multline*}
and
\begin{multline*}
S_2:= 2\pi i^{-k} \sum_{LM=N} \frac{\mu\left(L\right)}{L}\sum_{l|L^{\infty}}  \sum_{\substack{m<T^{1+\epsilon}\\\gcd\left(m,N\right)=1}} \frac{1}{ml} w\left(\frac{m}{T}\right) \\
\times \sum_{\substack{c\equiv 0\pmod{M}\\ c\neq ml}} \frac{S\left(m^2l^2,n;c\right)}{c}J_{k-1}\left(\frac{4\pi ml \sqrt{n}}{c}\right).
\end{multline*}
In what follows, we give an asymptotic formula for $S_1$, which is  the sum over the diagonal terms  $ml=c$ where $\gcd\left(m,N\right)=1$ and $l|L^{\infty}$.  Observe that the condition $ml = c$ can only be met if $M=1$ and  $L=N$ and so we have
\[
S\left(m^2l^2,n;c\right)=S\left(0,n;c\right)=\sum_{d|\gcd\left(c,n\right)} \mu\left(\frac{c}{d}\right)d.
\]
Hence,
\begin{align*}
S_1&=2\pi i^{-k}  J_{k-1}\left(4\pi\sqrt{n}\right) \frac{\mu\left(N\right)}{N}\sum_{l|N^{\infty}}  \sum_{\substack{m<T^{1+\epsilon}\\ \gcd\left(m,N\right)=1}} \frac{1}{\left(ml\right)^2}w\left(\frac{m}{T}\right) \sum_{d|\gcd\left(ml,n\right)} \mu\left(\frac{ml}{d}\right)d\\
&=2\pi i^{-k}  J_{k-1}\left(4\pi\sqrt{n}\right) \frac{\mu\left(N\right)}{N}  \sum_{l|N^{\infty}}\frac{\mu\left(l\right)}{l^2}   \left( \sum_{\substack{m<T^{1+\epsilon}\\ \gcd\left(m,N\right)=1}} \frac{1}{m^2}w\left(\frac{m}{T}\right) \sum_{d|\gcd\left(m,n\right)} \mu\left(\frac{m}{d}\right)d   \right)
\\
&=2\pi i^{-k} J_{k-1}\left(4\pi\sqrt{n}\right) \frac{\mu\left(N\right)}{N} \zeta_N\left(2\right)^{-1}\left( \sum_{d|n} \frac{1}{d} \sum_{ \substack{h<T^{1+\epsilon}/d \\ \gcd\left(h,N\right)=1}} \frac{1}{h^2} w\left(\frac{hd}{T}\right)\mu\left(h\right)    \right)
\\
&=2\pi i^{-k} J_{k-1}\left(4\pi\sqrt{n}\right) \frac{\mu\left(N\right)}{N} \frac{1}{\zeta\left(2\right)} \left(\frac{\sigma\left(n\right)}{n}+O\left(\frac{\sigma(n)}{T}\right)\right). \numberthis \label{main}
\end{align*}
 Next, we give an upper bound for $S_2$.  Let $\beta>0$ be some positive real number and  $S_{2,\beta}$ be  the same sum as $S_2$ but subjected to $K^{\beta}<l$, namely
\begin{multline*}
S_{2,\beta}:= 2\pi i^{-k}  \sum_{LM=N} \frac{\mu\left(L\right)}{L}\sum_{\substack{l>K^{\beta} \\ l|L^{\infty}}}  \sum_{\substack{m<T^{1+\epsilon} \\ \gcd\left(m,N\right)=1} } \frac{1}{ml}w\left(\frac{m}{T}\right) \\
 \times \sum_{\substack{c\equiv 0\pmod{M}\\ c\neq ml}} \frac{S\left(m^2l^2,n;c\right)}{c}J_{k-1}\left(\frac{4\pi ml \sqrt{n}}{c}\right).
\end{multline*}
Since $N$ is fixed and $S_1$  is supported on $l|N^{\infty}$ and $\mu\left(l\right)\neq 0$, it follows from \eqref{newpeter} that for sufficiently large $k$ (e.g., $K^{\beta}>N$),
$$
S_{2,\beta}=2\pi i^{-k}  \sum_{\substack{m<T^{1+\epsilon} \\ \gcd\left(m,N\right)=1}} \frac{1}{m}w\left(\frac{m}{T}\right) \sum_{LM=N} \frac{\mu\left(L\right)}{L}\sum_{\substack{l>K^{\beta} \\ l|L^{\infty}}} \frac{1}{l} \left( \Delta_{k,M}\left(m^2l^2,n\right)-\delta\left(ml^2,n\right)\right).
$$
By \cite[Corollary 2.2]{Luo}, we have
$$
\Delta_{k,M}\left(m^2l^2,n\right)-\delta\left(m^2l^2,n\right)=O_{N,\epsilon}\left(\frac{n^{\frac{1}{4}+\epsilon} \left(ml\right)^{\frac{1}{2}+\epsilon}}{k^{\frac{5}{6}}} \right).
$$
Therefore,
\[
S_{2,\beta} \ll_{N,\epsilon}  \sum_{\substack{m<T^{1+\epsilon} \\ \gcd\left(m,N\right)=1}} \frac{1}{m}w\left(\frac{m}{T}\right)\sum_{\substack{l>K^{\beta} \\ l|N^{\infty}}} \frac{1}{l}\frac{n^{\frac{1}{4}+\epsilon} \left(ml\right)^{\frac{1}{2}+\epsilon}}{k^{\frac{5}{6}}} .
\]
By \eqref{kn}, we have
\begin{equation}\label{S2b}
S_{2,\beta} \ll_{N,\epsilon} {k^{-\frac{1}{3}+\epsilon}} \sum_{m<T^{1+\epsilon}}   \sum_{\substack{l>K^{\beta} \\ l|N^{\infty}}}  \left(ml\right)^{-\frac{1}{2}+\epsilon} =O_{N,\epsilon}\left({T^{\frac{1}{2}}k^{-\frac{1}{3}-\frac{\beta}{2}+\epsilon}}\right).
\end{equation}
Finally, we give an upper bound for $S\left(\beta\right):=S_2-S_{2,\beta}$. We split $S\left(\beta\right)$ into two ranges, each of which has a restriction on the sum over $c\equiv 0 \pmod{M}$:
\begin{enumerate}
\item $2ml<c$,
\item $c< 2ml$ and $c\neq ml$,
\end{enumerate}
and we write $S_i\left(\beta\right)$ for the sum $S\left(\beta\right)$ subjected to the $i$th condition listed above. First, we give an upper bound for $S_1\left(\beta\right)$. Assume that $2ml<c$. Then by  \eqref{large}, \eqref{trans}, and \eqref{weil}, we have
\begin{align*}
&S_1\left(\beta\right) \\
&\ll \left| \sum_{LM=N} \frac{\mu\left(L\right)}{L} \sum_{\substack{l<K^{\beta}\\ l|L^{\infty}} } \sum_{\substack{m<T^{1+\epsilon}\\ \gcd\left(m,N\right)=1}} \frac{1}{ml}w\left(\frac{m}{T}\right)\sum_{ \substack{c>2ml\\  M|c}  } \frac{S\left(m^2l^2,n;c\right)}{c}J_{k-1}\left(\frac{4\pi ml \sqrt{n}}{c}\right) \right|\\
&\ll  \sum_{\substack{l<K^{\beta}\\ l|L^{\infty}}} \sum_{\substack{m<T^{1+\epsilon}\\ \gcd\left(m,N\right)=1}} \frac{1}{m l}w\left(\frac{m}{T}\right)\sum_{c>2ml } \left|\frac{S\left(m^2l^2,n;c\right)}{c}\right| \left|J_{k-1}\left(\frac{4\pi ml \sqrt{n}}{c}\right)\right|\\
&\ll  \sum_{h<K^{\beta}M^{1+\epsilon}} \frac{1}{h}\sum_{c>2h }  \left|\frac{e^{k\left(1-\frac{h}{c}+\log\left(\frac{h}{c}\right)\right)}}{k^{\frac{1}{3}}}\right| \ll  e^{-\left(0.19\right) k}. \numberthis \label{tailb}
\end{align*}
By inequalities~\eqref{lhsp},~\eqref{sqcont}, \eqref{splits12},  \eqref{main}, \eqref{S2b}, and \eqref{tailb}, we have
\begin{multline*}
\frac{12}{(k-1)N} \mathrm{Tr}~\mathcal{T}_n\left(N,k\right)^*=
 2\pi i^{-k} J_{k-1}\left(4\pi\sqrt{n}\right) \frac{\mu\left(N\right)}{N} \frac{1}{\zeta\left(2\right)}\frac{\sigma\left(n\right)}{n} + S_2\left(\beta\right)\\
 +O\left(\sigma(n)k^{-\frac{1}{3}} T^{-1}\right)+O_\epsilon\left(T^{-\frac{1}{2}}k^{-1+\epsilon}\mathcal{M}_1(k)+{T^{\frac{1}{2}}k^{-\frac{1}{3}-\frac{\beta}{2}+\epsilon}} \right).
\end{multline*}
We use $\sigma(n) \ll_\epsilon k^\epsilon$ to make the first error term $O_\epsilon \left(k^{-\frac{1}{3}+\epsilon} T^{-1}\right)$. We then multiply the above identity by $i^k=(-1)^{\frac{k}{2}}$ and take a smooth average by
\[
\frac{1}{K^{\delta}} \sum_{k>0, k\in 2 \mathbb{Z}} \psi \left(\frac{k-1-K}{K^{\delta}}\right),
\]
 yielding
\begin{multline}\label{finalform}
 \frac{1}{K^{\delta}} \sum_{k>0, k\in 2\mathbb{Z}} \psi \left(\frac{k-1-K}{K^{\delta}}\right)\frac{12 (-1)^{\frac{k}{2}}}{(k-1)N}\mathrm{Tr}~\mathcal{T}_n\left(N,k\right)^*\\
 =\pi J_{K}\left(4\pi\sqrt{n}\right) \frac{\mu\left(N\right)}{N} \frac{1}{\zeta\left(2\right)}\frac{\sigma\left(n\right)}{n}\left(1+o(1)\right)
 + \frac{1}{K^{\delta}} \sum_{k>0, k\in 2\mathbb{Z}} \psi \left(\frac{k-1-K}{K^{\delta}}\right) i^k S_2\left(\beta\right) \\ +O_\epsilon\left(T^{-\frac{1}{2}}K^{-1-\delta+\epsilon}\sum_{|k-1-K|<K^\delta} \mathcal{M}_1(k)\right)+O_\epsilon\left({T^{\frac{1}{2}}K^{-\frac{1}{3}-\frac{\beta}{2}+\epsilon}}+  T^{-1}K^{-\frac{1}{3}+\epsilon}\right),
\end{multline}
where we applied \eqref{sta3} in Lemma~\ref{avk} to the main term.

Next, we give an upper bound for the average of  $i^k S_2\left(\beta\right)$. We first have
\begin{multline}\label{S2sum}
  \frac{1}{K^{\delta}} \sum_{k>0, k\in 2\mathbb{Z}} \psi \left(\frac{k-1-K}{K^{\delta}}\right) i^k S_2\left(\beta\right)   =2\pi \sum_{LM=N} \frac{\mu\left(L\right)}{L}\sum_{\substack{l< K^{\beta}\\ l|L^{\infty}}}  \sum_{\substack{m<T^{1+\epsilon}\\ \gcd\left(m,N\right)=1}} \frac{1}{ml}w\left(\frac{m}{T}\right)  \\
\times \sum_{\substack{c<2ml \\ M|c,~c\neq ml}  } \frac{S\left(m^2l^2,n;c\right)}{c} \frac{1}{K^{\delta}} \sum_{k>0, k\in 2\mathbb{Z}} \psi \left(\frac{k-1-K}{K^{\delta}}\right)J_{k-1}\left(\frac{4\pi ml \sqrt{n}}{c}\right).
\end{multline}
Let $x:= \frac{4\pi ml \sqrt{n}}{c}$. Then we have $x > 2\pi \sqrt{n} \gg K$, because we assumed that $K-4\pi\sqrt{n}=o\left(n^{\frac{1}{6}}\right)$. Let $\eta>1-3\delta>0$ be a constant to be chosen later. Note that
\[
|x-K|<K^{\eta+\delta}
\]
implies that
\[
\left|\frac{ml}{c} - 1\right| < 2K^{\eta+\delta-1}.
\]
We assume that $\eta$ is chosen sufficiently close to $1-3\delta$ so that the exponent $\eta+\delta-1$ is negative. In order to apply Lemma \ref{avk}, we now split the sum \eqref{S2sum} into two ranges
 \begin{enumerate}
\item $c< 2ml$ and  $ |\frac{ml}{c}-1| > 2K^{\eta+\delta-1}$, and
\item $c< 2ml$ and $ |\frac{ml}{c}-1| < 2K^{\eta+\delta-1}$,
\end{enumerate}
each of which has a restriction on the sum over $c\equiv 0 \pmod{M}$. We denote the sums by $S_{2,1}$ and $S_{2,2}$ respectively, so that \eqref{S2sum} is equal to $S_{2,1}+S_{2,2}$. By \eqref{sta1}, \eqref{S2sum}, and  \eqref{weil}, we have
 \begin{multline}\label{S21}
S_{2,1}  \ll_{A,\eta,\delta}  \sum_{\substack{l<K^{\beta}\\ l|N^{\infty}}}  \sum_{\substack{m<T^{1+\epsilon}\\ \gcd\left(m,N\right)=1}} \frac{1}{ml}  \sum_{ c<2ml } \left|\frac{S\left(m^2l^2,n;c\right)}{c}\right|
 K^{-A}
 \\
  \ll_{A,\eta,\delta,\epsilon} K^{-A} \sum_{\substack{l<K^{\beta}\\ l|N^{\infty}}}  \sum_{\substack{m<T^{1+\epsilon} \\ \gcd\left(m,N\right)=1}} \frac{1}{ml}  \sum_{ c<2ml }  \sqrt{\gcd\left(m,n,c\right)} c^{-\frac{1}{2}+\epsilon}
 \\
 \ll_{A,\eta,\delta,\epsilon} T^{\frac{1}{2}+\epsilon}K^{-A}.
\end{multline}
For $S_{2,2}$, we apply \eqref{weil} and \eqref{sta2}, yielding
 \begin{align*}
S_{2,2}&\ll  \sum_{\substack{l<K^{\beta}\\ l|N^{\infty}}}  \sum_{\substack{m<T^{1+\epsilon} \\ \gcd\left(m,N\right)=1}} \frac{1}{ml}  \sum_{  \substack{|\frac{ml}{c}-1| < 2K^{\eta+\delta-1}\\c\neq ml} } \left|\frac{S\left(m^2l^2,n;c\right)}{c}\right| K^{-\frac{1}{3}}
\\
&\ll_{\epsilon,N} K^{-\frac{1}{3}}  \sum_{\substack{l<K^{\beta}\\ l|N^{\infty}}}  \sum_{m<T^{1+\epsilon}} \frac{1}{ml} \sqrt{\gcd\left(m,n\right)} \sum_{  \substack{|\frac{ml}{c}-1| < 2K^{\eta+\delta-1}\\c\neq ml} }  c^{-\frac{1}{2}+\epsilon}
\\
&\ll_{\epsilon,N} K^{-\frac{1}{3}}  \sum_{\substack{l<K^{\beta}\\ l|N^{\infty}}}  \sum_{m<T^{1+\epsilon}} \frac{1}{ml} \sqrt{\gcd\left(m,n\right)} \frac{\left(ml\right)^{\frac{1}{2}+\epsilon}}{K^{1-\eta-\delta}} \\
&\ll_{\epsilon,N} K^{-\frac{1}{3}}  \sum_{\substack{l<K^{\beta}\\ l|N^{\infty}}}  \sum_{m<T^{1+\epsilon}}  \sqrt{\gcd\left(m,n\right)} \frac{\left(ml\right)^{-\frac{1}{2}+\epsilon}}{K^{1-\eta-\delta}} \\
 &\ll_{\epsilon,N} T^{\frac{1}{2}+\epsilon}K^{-\frac{1}{3}-1+\eta+\delta}. \numberthis \label{S22}
\end{align*}
Therefore, by inequalities \eqref{finalform}, \eqref{S21}, and \eqref{S22}, we have
\begin{multline*}
 \frac{1}{K^{\delta}} \sum_{k>0, k\in 2\mathbb{Z}} \psi \left(\frac{k-1-K}{K^{\delta}}\right)\frac{12 (-1)^{\frac{k}{2}}}{(k-1)N}\mathrm{Tr}~\mathcal{T}_n\left(N,k\right)^*\\
 =\pi J_{K}\left(4\pi\sqrt{n}\right) \frac{\mu\left(N\right)}{N} \frac{1}{\zeta\left(2\right)}\frac{\sigma\left(n\right)}{n}\left(1+o(1)\right)\\ +O_{A,\eta,\delta,\epsilon}\left(T^{-\frac{1}{2}}K^{-1-\delta+\epsilon}\sum_{|k-K|<K^\delta} \mathcal{M}_1(k)+{T^{\frac{1}{2}}K^{-\frac{1}{3}-\frac{\beta}{2}+\epsilon}}+T^{\frac{1}{2}+\epsilon}K^{-\frac{1}{3}-1+\eta+\delta}+T^{-1}K^{-\frac{1}{3}+\epsilon} \right).
 \end{multline*}
 In order to bound the contribution from $\sum \mathcal{M}_1(k)$, we recall from \cite{MR1990480} that
\[
\sum_{|k-K|<K^\theta} \int_{-(\log k)^2}^{(\log k)^2}  \sum_{f\in B_{k,N}^* } \left| L\left(\frac{1}{2}+it,\mathrm{sym}^2 f \right)\right|^2 dt  \ll_{\epsilon,\theta} K^{1+\theta+\epsilon}
\]
 provided that $\theta>1/3$. This in particular implies that
\begin{multline*}
\left(\sum_{|k-K|<K^\delta} \mathcal{M}_1(k)\right)^2 \\
\leq \sum_{|k-K|<K^\delta}\sum_{f\in B_{k,N}^* } \int_{-(\log k)^2}^{(\log k)^2}   dt\sum_{|k-K|<K^\theta} \int_{-(\log k)^2}^{(\log k)^2}  \sum_{f\in B_{k,N}^* } \left| L\left(\frac{1}{2}+it,\mathrm{sym}^2 f \right)\right|^2 dt\\
\ll_{\epsilon,\theta} K^{2+\theta+\delta+\epsilon},
\end{multline*}
by the Cauchy--Schwarz inequality. We therefore have
\[
\sum_{|k-K|<K^\delta} \mathcal{M}_1(k) \ll_\epsilon K^{\frac{7}{6}+\frac{\delta}{2}+\epsilon},
\]
and so by choosing $\beta$ large enough, $T = K^{\frac{1}{2}+\frac{3}{2}\delta}$, and $\eta=1-3\delta+\epsilon$, we conclude that
\begin{multline*}
 \frac{1}{K^{\delta}} \sum_{k>0, k\in 2\mathbb{Z}}\frac{1}{k-1} \psi \left(\frac{k-1-K}{K^{\delta}}\right)\frac{12 (-1)^{\frac{k}{2}}}{N}\mathrm{Tr}~\mathcal{T}_n\left(N,k\right)^*\\
 =\pi J_{K}\left(4\pi\sqrt{n}\right) \frac{\mu\left(N\right)}{N} \frac{1}{\zeta\left(2\right)}\frac{\sigma\left(n\right)}{n}\left(1+o(1)\right) +O_{\epsilon,\delta}(K^{-\frac{1}{12}-\frac{5}{4}\delta+\epsilon}).
 \end{multline*}
In order to complete the proof, note that
\[
\frac{1}{k-1} - \frac{1}{K} = \frac{K-k+1}{(k-1)K} = O(K^{\delta-2}),
\]
and that
\[
\mathrm{Tr}~\mathcal{T}_n\left(N,k\right)^* \ll \sigma(n) K.
\]
So the error that occurs when replacing $\frac{1}{k-1}$ by $\frac{1}{K}$ in the left hand side of the equation is
\[
\ll \sigma(n)K^{\delta-1}.
\]
Assuming that $\frac{1}{5}<\delta <\frac{1}{3}$ and  rearranging lead to the final expression in Theorem \ref{noweight}.
\end{proof}

\subsection{Proof of Theorem~\ref{rmwe}}
\begin{proof}
The method of the proof is similar to the proof of Theorem~\ref{weylpeter}. Let $U_{n}\left(x\right)$ be the $n$th Chebyshev polynomial of the second
kind. A quick computation shows that
\[
\int_{-2}^{2} U_n\left(\frac{x}{2}\right) d\mu_p\left(x\right)=\left\{\begin{array}{cl}  \frac{1}{p^{\frac{n}{2}}}& \text{if }n \text{ is  even}\\ 0 & \text{otherwise.}\end{array}\right.
\]
By Theorem~\ref{noweight}, there exists  $k_n\in  [\lfloor 4\pi \sqrt{p^n} \rfloor - p^{\frac{n}{6}}, \lfloor 4\pi \sqrt{p^n} \rfloor +p^{\frac{n}{6}} ] $ such that
\[
  \int_{-2}^{2} U_n\left(\frac{x}{2}\right) d\mu_p\left(x\right)-\int_{-2}^{2} U_n\left(\frac{x}{2}\right) d\mu_{k_n,N}^* \gg k_n^{-\frac{1}{3}}.
\]
By the above inequality  and by integration by parts with the upper bound $|U_n'(x)|\ll n^2$, we have
\[
D\left(\mu_{k_n,N}^*, \mu_{p}\right)  \gg  \frac{1}{n^2 k_n^{\frac{1}{3}}}.
\]
We complete the proof of Theorem \ref{rmwe} by observing that $n \ll \log k_n  $.
\end{proof}

\section{The Eichler--Selberg trace formula}\label{selberg}
The main purpose of this section is to prove Theorem~\ref{selbergmain1} and Theorem~\ref{selbergmain2}. We first recall the Eichler--Selberg trace formula. We use the version from  \cite[Theorem 10]{MS} (see also \cite{serre}).
\begin{theorem}[The Eichler--Selberg trace formula] \label{ES}
For every positive integer $n \geq 1$, the trace $\mathrm{Tr}$ of $\mathcal{T}_n=\mathcal{T}_n\left(k,N\right)$ acting on $S_k\left(N\right)$ is given by
\[
\mathrm{Tr}~ \mathcal{T}_n = A_1\left(n,k,N\right) + A_2\left(n,k,N\right) + A_3\left(n,k,N\right) + A_4\left(n,k\right),
\]
where $A_i\left(n,k\right)$'s are as follows:
\[
A_1\left(n,k,N\right) = \left\{\begin{array}{cl}  \frac{k-1}{12} \nu \left(N\right)\frac{1}{\sqrt{n}}& \text{if }n \text{ is a square}\\ 0 & \text{otherwise}\end{array}\right. \text{ where } \nu\left(N\right) = N\prod_{p|N} \left(1+\frac{1}{p}\right).
\]
\[
A_2\left(n,k,N\right)= -\frac{1}{2}n^{-\frac{k-1}{2}} \sum_{t\in \mathbb{Z}, ~ t^2<4n} \frac{\rho_{t,n}^{k-1} - \bar{\rho}_{t,n}^{k-1}}{\rho_{t,n} - \bar{\rho}_{t,n}} \sum_{f} h_w \left(\frac{t^2-4n}{f^2}\right) \mu\left(t,f,n,N\right),
\]
where $\rho_{t,n}$ and $\bar{\rho}_{t,n}$ are zeros of $x^2-tx+n$, and the inner sum runs over all positive divisors $f$ of $t^2-4n$ such that $\left(t^2-4n\right)/f^2 \in \mathbb{Z}$ is congruent to $0$ or $1 \pmod{4}$. The function $\mu\left(t,f,n,N\right)$ is given by
\[
\mu\left(t,f,n,N\right) = \frac{\nu\left(N\right)}{\nu\left(N/N_f\right)} M\left(t,n,NN_f\right),
\]
where $N_f = \gcd\left(N,f\right)$ and $M\left(t,n,K\right)$ denotes the number of solutions of the congruence $x^2 -tx+n\equiv 0 \pmod{K}$. Next,
\[
A_3\left(n,k,N\right) = - n^{-\frac{k-1}{2}} \sum_{d|n,~0 <d \leq \sqrt{n}} d^{k-1} \sum_{c|N, \gcd\left(c,\frac{N}{c}\right)| \gcd\left(N, \frac{n}{d}-d\right)} \varphi \left(\gcd\left(c,\frac{N}{c}\right)\right).
\]
Here, $\varphi$ is Euler's totient function, and in the first summation, if there is a contribution from the term $d=\sqrt{n}$, it should be multiplied by $\frac{1}{2}$. Finally,
\[
A_4\left(n,k\right) = \left\{\begin{array}{cl}n^{-\frac{1}{2}} \sum_{t|n} t & \text{if } k=2, \\0 &\text{otherwise.}\end{array}\right.
\]
\end{theorem}
To relate the trace of $\mathcal{T}_n$ acting on $S_k\left(N\right)$ and the trace of its restriction $\mathcal{T}_n^*$ to $S_k\left(N\right)^*$, one may use Atkin--Lehner decomposition for squarefree integers $N$ to derive (see for instance, \cite[Equation (2)]{MR1604056})
\[
\mathrm{Tr}~\mathcal{T}_n\left(k,N\right) = \sum_{d|N} \sigma\left(N/d\right) \mathrm{Tr}~\mathcal{T}_n^*\left(d,k\right),
\]
and by M\"obius inversion, this implies that
\begin{equation}\label{mobius}
\mathrm{Tr}~\mathcal{T}_n^*\left(N,k\right) = \sum_{d|N} \sigma\left(N/d\right)\mu\left(N/d\right)\mathrm{Tr}~\mathcal{T}_n\left(d,k\right).
\end{equation}
Therefore we have the following.
\begin{lemma}\label{selbergnew}
Assume that $N$ is a squarefree integer. For every positive integer $n \geq 1$, the trace $\mathrm{Tr}$ of $\mathcal{T}_n=\mathcal{T}_n\left(k,N\right)$ restricted to $S_k\left(N\right)^*$ is given by
\[
\mathrm{Tr}~ \mathcal{T}_n^* = B_1\left(n,k,N\right) + B_2\left(n,k,N\right) + B_3\left(n,k,N\right) + B_4\left(n,k,N\right),
\]
where $B_i\left(n,k\right)$'s are as follows:
\[
B_1\left(n,k,N\right) = \left\{\begin{array}{cl}  \frac{k-1}{12} \varphi \left(N\right)\frac{1}{\sqrt{n}}& \text{if }n \text{ is a square,}\\ 0 & \text{otherwise.}\end{array}\right.
\]
\[
B_2\left(n,k,N\right)= -\frac{1}{2}n^{-\frac{k-1}{2}} \sum_{t\in \mathbb{Z}, ~ t^2<4n} \frac{\rho_{t,n}^{k-1} - \bar{\rho}_{t,n}^{k-1}}{\rho_{t,n} - \bar{\rho}_{t,n}} \sum_{f} h_w \left(\frac{t^2-4n}{f^2}\right) \tilde{\mu}\left(t,f,n,N\right),
\]
where $\rho_{t,n}$ and $\bar{\rho}_{t,n}$ are zeros of $x^2-tx+n$, and the inner sum runs over all positive divisors of $t^2-4n$ such that $\left(t^2-4n\right)/f^2 \in \mathbb{Z}$ is congruent to $0$ or $1 \pmod{4}$. The function $\tilde{\mu}\left(t,f,n,N\right)$ is given by
\[
\tilde{\mu}\left(t,f,n,N\right) = \sum_{d|N} \sigma\left(N/d\right)\mu\left(N/d\right)\mu\left(t,f,n,d\right).
\]
\[
B_3\left(n,k,N\right) =  \left\{\begin{array}{cl}  - n^{-\frac{k-1}{2}} \sum_{d|n,~0 <d \leq \sqrt{n}} d^{k-1} & \text{if }N=1,\\ 0 & \text{otherwise.}\end{array}\right.
\]
In the first summation, if there is a contribution from the term $d=\sqrt{n}$, it should be multiplied by $\frac{1}{2}$.
\[
B_4\left(n,k,N\right) = \left\{\begin{array}{cl}\mu\left(N\right)n^{-\frac{1}{2}} \sum_{t|n} t & \text{if } k=2, \\0 &\text{otherwise.}\end{array}\right.
\]
\end{lemma}
\begin{proof}
This follows from Theorem \ref{ES} and \eqref{mobius}.
\end{proof}

\subsection{Analytic setup}
Let $\phi$ be a positive even rapidly decaying function whose Fourier transform $\hat{\phi}$ is supported in $\left\lbrack-\frac{1}{100},\frac{1}{100}\right\rbrack$. In this section, we study the second moment of $B_2$:
\begin{equation}\label{sufficient}
\sum_{k>0, k\in 2\mathbb{Z}} \phi\left(\frac{k-1}{T}\right) \left|B_2\left(n,k,N\right)\right|^2 = \frac{1}{2}\sum_{ k\in 2\mathbb{Z}} \phi\left(\frac{k-1}{T}\right) \left|B_2\left(n,k,N\right)\right|^2,
\end{equation}
where we used $B_2\left(n,k,N\right) = -B_2\left(n,2-k,N\right)$.

We first collect some preliminary estimates.
\begin{lemma}\label{estimates}
We have
\begin{equation}\label{dim}
|S_k\left(N\right)^*|= \frac{k-1}{12}\varphi\left(N\right) + O_N\left(1\right),
\end{equation}
and
\begin{equation}\label{A_2}
B_2\left(n,k,N\right) \ll_N \sigma_1\left(n\right).
\end{equation}
\end{lemma}
\begin{proof}
The asymptotic \eqref{dim} follows from  \cite[Theorem 13]{MS} and \eqref{mobius}.

To prove \eqref{A_2}, note that
\[
\left|n^{-\frac{k-1}{2}} \frac{\rho_{t,n}^{k-1} - \bar{\rho}_{t,n}^{k-1}}{\rho_{t,n} - \bar{\rho}_{t,n}}\right| \leq \frac{2}{|\rho_{t,n} - \bar{\rho}_{t,n}|} =\frac{2}{\sqrt{4n-t^2}} \leq 2.
\]
Therefore
\[
|B_2\left(n,k,N\right)| \leq 2\sum_{t^2<4n}\sum_{f} h_w \left(\frac{t^2-4n}{f^2}\right) \tilde{\mu}\left(t,f,n,N\right) \ll_N \sigma_1\left(n\right),
\]
where we combined Lemma 16 \cite{MS} and the trivial upper bound $\tilde{\mu}\left(t,f,n,N\right) \ll_N 1$ in the last estimate.
\end{proof}

For $t\in \mathbb{Z}$ such that $t^2<4n$, define $0<\theta_{t,n}<\pi$ by
\[
\sqrt{n}e^{i\theta_{t,n}} = \frac{1}{2}\left(t+i\sqrt{4n-t^2}\right).
\]
We record some trivial estimates regarding $\theta_{t,n}$'s.
\begin{lemma}\label{theta}
For an integer $t$ such that $t^2 <n$, we have
\[
\pi-\frac{1}{2\sqrt{n}}\theta_{t,n} \geq \frac{1}{2\sqrt{n}},
\]
and
\[
\theta_{t,n} - \theta_{t+1,n} \geq \frac{1}{2\sqrt{n}}.
\]
\end{lemma}
\begin{proof}
We have
\[
\sin \theta_{t,n} = \frac{\sqrt{4n-t^2}}{2\sqrt{n}} \geq \frac{1}{2\sqrt{n}}.
\]
Also,
\[
e^{i\left(\theta_{t,n} - \theta_{t+1,n}\right)} = \frac{1}{4n} \left(t+i\sqrt{4n-t^2}\right)\left(t+1-i\sqrt{4n-\left(t+1\right)^2}\right),
\]
so
\begin{align*}
\sin \left(\theta_{t,n} - \theta_{t+1,n}\right) &= \frac{1}{4n} \left(\left(t+1\right)\sqrt{4n-t^2}-t\sqrt{4n-\left(t+1\right)^2}\right)\\
&=\frac{1}{4n} \frac{\left(t+1\right)^2\left(4n-t^2\right)-t^2\left(4n-\left(t+1\right)^2\right)}{\left(t+1\right)\sqrt{4n-t^2}+t\sqrt{4n-\left(t+1\right)^2}}\\
&= \frac{2t+1}{\left(t+1\right)\sqrt{4n-t^2}+t\sqrt{4n-\left(t+1\right)^2}}\\
&\geq \frac{1}{\sqrt{4n}}. \qedhere
\end{align*}
\end{proof}
We define $D_N\left(t,n\right)$ by
\[
D_N\left(t,n\right) = \frac{i}{2\sqrt{4n-t^2}} \sum_{f} h_w \left(\frac{t^2-4n}{f^2}\right) \tilde{\mu}\left(t,f,n,N\right),
\]
where the inner sum runs over all positive divisors $f$ of $t^2-4n$ such that $\left(t^2-4n\right)/f^2 \in \mathbb{Z}$ is congruent to $0$ or $1 \pmod{4}$. Then we may write $B_2\left(n,k,N\right)$ as
\[
B_2\left(n,k,N\right) = \sum_{t\in \mathbb{Z}, ~ t^2<4n} \left(e^{i\left(k-1\right)\theta_{t,n}}-e^{-i\left(k-1\right)\theta_{t,n}}\right)D_N\left(t,n\right).
\]
Then expanding \eqref{sufficient} and using $D_N\left(t,n\right)=-D_N\left(-t,n\right)$, we get
\begin{align*}
&\sum_{k\in 2\mathbb{Z}} \phi\left(\frac{k-1}{T}\right) \left|B_2\left(n,k,N\right)\right|^2 \\
=& 4\sum_{k\in 2\mathbb{Z}} \phi\left(\frac{k-1}{T}\right) \sum_{t^2<4n}|D_N\left(t,n\right)|^2  \\
+& \sum_{t_1\neq t_2}\sum_{k\in 2\mathbb{Z}} \phi\left(\frac{k-1}{T}\right) e^{\pm i\left(k-1\right)\left(\theta_{t_1,n}- \theta_{t_2,n}\right)} D_N\left(t_1,n\right) D_N\left(t_2,n\right)\\
-&\sum_{t_1\neq -t_2}\sum_{k\in 2\mathbb{Z}} \phi\left(\frac{k-1}{T}\right) e^{\pm i\left(k-1\right)\left(\theta_{t_1,n}+ \theta_{t_2,n}\right)} D_N\left(t_1,n\right) D_N\left(t_2,n\right)\\
=&D+OD, \numberthis \label{OD1}
\end{align*}
where the diagonal part $D$ comes from $\theta_{t_1,n}+\theta_{t_2,n} = \pi$ and from $\theta_{t_1,n}=\theta_{t_2,n}$, and the off diagonal part $OD$ amounts to remaining terms. Note from Lemma \ref{theta} that, unless it is an integer multiple of $\pi$, $\theta_{t_1,n}\pm \theta_{t_2,n}$ are contained in $\left\lbrack \frac{1}{2\sqrt{n}},\pi -\frac{1}{2\sqrt{n}}\right\rbrack$ modulo $\pi$. Therefore we have
\begin{align*}
OD&\ll \sup_{\theta \in \left\lbrack \frac{1}{2\sqrt{n}},\pi -\frac{1}{2\sqrt{n}}\right\rbrack}\left|\sum_{k\in 2\mathbb{Z}} \phi\left(\frac{k-1}{T}\right) e^{i\left(k-1\right)\theta} \right| \sum_{t_1,t_2} |D_N\left(t_1,n\right) D_N\left(t_2,n\right)|\\
&\ll_N  \sup_{\theta \in \left\lbrack \frac{1}{2\sqrt{n}},\pi -\frac{1}{2\sqrt{n}}\right\rbrack}\left|\sum_{k\in 2\mathbb{Z}} \phi\left(\frac{k-1}{T}\right) e^{ i\left(k-1\right)\theta}\right| \sigma_1\left(n\right)^2. \numberthis \label{OD2}
\end{align*}

\begin{lemma}\label{idphi}
Let $T\geq \sqrt{n}$. Then for any $\theta$ that satisfies
$
\theta \in \left\lbrack \frac{1}{2\sqrt{n}},\pi -\frac{1}{2\sqrt{n}}\right\rbrack,
$
we have
\[
\sum_{k\in 2\mathbb{Z}} \phi\left(\frac{k-1}{T}\right) e^{i\left(k-1\right)\theta}= 0,
\]
and as a result
\[
\sum_{k\in 2\mathbb{Z}} \phi\left(\frac{k-1}{T}\right) \left|B_2\left(n,k,N\right)\right|^2
= 4\sum_{k\in 2\mathbb{Z}} \phi\left(\frac{k-1}{T}\right) \sum_{t^2<4n}|D_N\left(t,n\right)|^2.
\]
\end{lemma}
\begin{proof}
From the Poisson summation formula we have
\begin{equation}\label{lemeq1}
\sum_{k\in 2\mathbb{Z}} \phi\left(\frac{k-1}{T}\right) e^{i\left(k-1\right)\theta} =\sum_{n\in\mathbb{Z}} \phi\left(\frac{2n-1}{T}\right) e^{i\left(2n-1\right)\theta}  = \sum_{m\in \mathbb{Z}} \Phi \left(m\right),
\end{equation}
where
\[
\Phi\left( y \right) =\frac{T}{2}e^{-\pi i y}\hat{\phi}\left(\frac{T\left(\pi y-\theta\right)}{2\pi}\right).
\]
In the last expression, for any $m\in \mathbb{Z}$, we have
\[
\left|\frac{T\left(\pi m-\theta\right)}{2\pi}\right| \geq \frac{1}{4\pi},
\]
and since $\hat{\phi}$ is assumed to be supported in $\left\lbrack -\frac{1}{100}, \frac{1}{100}\right\rbrack$, the right-hand side of \eqref{lemeq1} vanishes.
\end{proof}
We are ready to prove the following.
\begin{lemma}\label{sufficient11} Let $N> 1$ be a fixed square-free integer.
Let $\phi$ be a positive even rapidly decaying function whose Fourier transform $\hat{\phi}$ is supported in $\left\lbrack-\frac{1}{100},\frac{1}{100}\right\rbrack$. Let $T\geq \sqrt{n}$. Then we have
\begin{multline}\label{sufficient1}
\sum_{k>0, k\in 2\mathbb{Z}} \phi\left(\frac{k-1}{T}\right) \left|\mathrm{Tr}~ \mathcal{T}_n^*- \frac{k-1}{12} \varphi \left(N\right)\frac{\delta\left(n,\square\right)}{\sqrt{n}}\right|^2\\
= 2\sum_{k\in 2\mathbb{Z}} \phi\left(\frac{k-1}{T}\right) \sum_{t^2<4n}|D_N\left(t,n\right)|^2 - \phi\left(\frac{1}{T}\right)\frac{\sigma_1\left(n\right)^2}{n}   + O_{\epsilon}\left(n^{\frac{1}{2}+\epsilon}\right).
\end{multline}
\end{lemma}
\begin{proof}
By Lemma~\ref{selbergnew}, for $N> 1$ we have
\[
B_3\left(n,k,N\right)=0.
\]
The summand of the left-hand side of \eqref{sufficient1} agrees with $B_2\left(n,k,N\right)$ unless $k=2$, so from \eqref{sufficient}, \eqref{OD1}, \eqref{OD2}, and Lemma \ref{idphi}, we have
\begin{multline*}
\sum_{k>0, k\in 2\mathbb{Z}} \phi\left(\frac{k-1}{T}\right) \left|\mathrm{Tr}~ \mathcal{T}_n^*- \frac{k-1}{12} \varphi \left(N\right)\frac{\delta\left(n,\square\right)}{\sqrt{n}}\right|^2\\
= 2\sum_{k\in 2\mathbb{Z}} \phi\left(\frac{k-1}{T}\right) \sum_{t^2<4n}|D_N\left(t,n\right)|^2\\
+  \phi\left(\frac{1}{T}\right)\left( \left|\mathrm{Tr}~ \mathcal{T}_n^*- \frac{1}{12} \varphi \left(N\right)\frac{\delta\left(n,\square\right)}{\sqrt{n}}\right|^2 - |B_2\left(n,2,N\right)|^2\right) + O_{\epsilon}\left(n^{\frac{1}{2}+\epsilon}\right).
\end{multline*}
By Lemma~\ref{selbergnew}, for $k=2$ and $N>1 $ we have
\[
B_2\left(n,2,N\right)=\mathrm{Tr}~ \mathcal{T}_n^*- \frac{1}{12} \varphi \left(N\right)\frac{\delta\left(n,\square\right)}{\sqrt{n}} - \mu\left(N\right)\frac{\sigma_1\left(n\right)}{\sqrt{n}}.
\]
By  the Ramanujan bound for weight $2$ modular forms, we have
\[
\mathrm{Tr}~ \mathcal{T}_n^* \ll_{\epsilon,N} n^\epsilon.
\]
Hence,
\[
 \left|\mathrm{Tr}~ \mathcal{T}_n^*- \frac{1}{12} \varphi \left(N\right)\frac{\delta\left(n,\square\right)}{\sqrt{n}}\right|^2 - |B_2\left(n,2,N\right)|^2=-\frac{\sigma_1\left(n\right)^2}{n}   + O_{\epsilon}\left(n^{\frac{1}{2}+\epsilon}\right).\qedhere
\]
\end{proof}

\subsection{Arithmetic sum}
In this section, we estimate the arithmetic part of \eqref{sufficient1}:
\[
\sum_{t^2<4n}|D_N\left(t,n\right)|^2.
\]
\begin{theorem}\label{arith}
Assume that $n$ is odd. Then we have
\[
\sqrt{n}\ll_N \sum_{t^2<4n}|D_N\left(t,n\right)|^2\ll_N \sqrt{n} \left(\log n\right)^2 \left(\log \log n\right)^4.
\]
\end{theorem}

Recall that
\[
D_N\left(t,n\right) = \frac{i}{2\sqrt{4n-t^2}} \sum_{f} h_w \left(\frac{t^2-4n}{f^2}\right) \tilde{\mu}\left(t,f,n,N\right),
\]
where the inner sum runs over all positive divisors $f$ of $t^2-4n$ such that $\left(t^2-4n\right)/f^2 \in \mathbb{Z}$ is congruent to $0$ or $1 \pmod{4}$. $\tilde{\mu}\left(t,f,n,N\right)$ is given by
\[
\tilde{\mu}\left(t,f,n,N\right) = \sum_{d|N} \sigma\left(N/d\right)\mu\left(N/d\right)\mu\left(t,f,n,d\right),
\]
and $\mu\left(t,f,n,N\right)$ is given by
\[
\mu\left(t,f,n,N\right) = \frac{\nu\left(N\right)}{\nu\left(N/N_f\right)} M\left(t,n,NN_f\right),
\]
where $N_f = \gcd\left(N,f\right)$ and $M\left(t,n,K\right)$ denotes the number of solutions of the congruence $x^2 -tx+n\equiv 0 \pmod{K}$.

Denote by $H\left(n\right) = \sum_{f^2|n} h_w\left(-n/f^2\right)$ the Hurwitz class number. For the upper bound for the arithmetic sum, we write
\begin{equation}\label{upper}
\sum_{t^2<4n}D_N\left(t,n\right)^2 \ll_N \sum_{t^2<4n}\frac{1}{4n-t^2} H^2 \left(t^2-4n\right),
\end{equation}
using the estimate $\mu\left(t,f,n,N\right) \ll_N 1$.

For the lower bound, we first prove the following.
\begin{lemma}
Assume that $n$ is odd. Fix an odd integer $0<n_0 <2N$ such that $\left(\frac{n_0^2-4n}{p}\right)=-1$ for all odd primes $p|N$. Then $\tilde{\mu}\left(t,f,n,N\right)=\sigma\left(N\right)\mu\left(N\right)$ for any $t \equiv n_0\pmod{2N}$.
\end{lemma}
\begin{proof}
For such $t$, we have $\mu\left(t,f,n,d\right)=0$ unless $d=1$ or $2$. So for an odd $N$,
\[
\tilde{\mu}\left(t,f,n,N\right) = \sigma\left(N\right)\mu\left(N\right).
\]
When $N$ is even, we have
\begin{multline*}
\tilde{\mu}\left(t,f,n,N\right) = \sigma\left(N\right)\mu\left(N\right) + \sigma\left(N/2\right)\mu\left(N/2\right)\mu\left(t,f,n,2\right) \\
= \sigma\left(N/2\right)\mu\left(N/2\right)\left(\mu\left(t,f,n,2\right) - 2\right),
\end{multline*}
where
\[
\mu\left(t,f,n,2\right) =  M\left(t,n,2\right),
\]
because $\gcd\left(N,f\right)|\gcd\left(N,t^2-4n\right) =1$. Then $M\left(t,n,2\right) = 0$ since both $n$ and $t$ are assumed to be odd, and therefore
\[
\tilde{\mu}\left(t,f,n,N\right)=\sigma\left(N/2\right)\mu\left(N/2\right)\times \left(- 2\right) = \sigma\left(N\right)\mu\left(N\right).\qedhere
\]
\end{proof}
Using this lemma, we bound the arithmetic sum from the below under the assumption that $n$ is odd as follows:
\begin{multline}\label{lower}
\sum_{t^2<4n}D_N\left(t,n\right)^2 \geq \sum_{\substack{t^2<4n\\~t\equiv n_0 \pmod{2N}}} D_N\left(t,n\right)^2 =\sum_{\substack{t^2<4n\\~t\equiv n_0 \pmod{2N}}}\frac{\sigma\left(N\right)^2}{4n-t^2} H^2 \left(t^2-4n\right)\\
 \geq \sum_{\substack{t^2<4n\\~t\equiv n_0 \pmod{2N}}}\frac{1}{4n-t^2} H^2 \left(t^2-4n\right).
\end{multline}
We now handle the right-hand sides of \eqref{upper} and \eqref{lower} separately.
\subsubsection{Upper bound}
We first recall from \cite[p.273, c)]{Cohen} that for $n=Df^2<0$,
\begin{equation}\label{hur}
H\left(n\right) = \frac{h\left(D\right)}{w\left(D\right)} \sum_{d|f} \mu \left(d\right) \chi_D\left(d\right) \sigma_1 \left(\frac{f}{d}\right),
\end{equation}
where $2w\left(D\right)$ is the number of units in $\mathbb{Q}\left(\sqrt{-D}\right)$. Note that
\[
\sum_{d|f} \mu \left(d\right) \chi_D\left(d\right) \sigma_1 \left(\frac{f}{d}\right)
\]
is multiplicative in $f$, and
\begin{multline*}
\sum_{d|p^k} \mu \left(d\right) \chi_D\left(d\right) \sigma_1 \left(\frac{p^k}{d}\right) = \sigma_1 \left(p^k\right) - \chi_D\left(p\right)\sigma_1 \left(p^{k-1}\right) \\
\leq \sigma_1 \left(p^k\right)+\sigma_1 \left(p^{k-1}\right)< \left(1+\frac{1}{p}\right) \sigma_1 \left(p^k\right).
\end{multline*}
Therefore
\[
\sum_{d|f} \mu \left(d\right) \chi_D\left(d\right) \sigma_1 \left(\frac{f}{d}\right) < \sigma_1 \left(f\right) \prod_{p|f}\left(1+\frac{1}{p}\right) \ll f  \left(\log \log f\right)^2,
\]
where we used Gr\"onwall's theorem in the last inequality. Using a standard upper bound $h\left(D\right)\ll \sqrt{D} \log D$ yields
\[
H\left(n\right) \ll \sqrt{D}f \log D  \left(\log\log f\right)^2  \ll \sqrt n \log n \left(\log \log n\right)^2.
\]
Now we apply this to \eqref{upper} to conclude that
\[
\sum_{t^2<4n}D_N\left(t,n\right)^2 \ll_N \sqrt n \left(\log n\right)^2\left(\log \log n\right)^4.
\]
\subsubsection{Lower bound}
From the Cauchy--Schwarz inequality,
\begin{multline*}
\sum_{\substack{t^2<4n\\ t\equiv n_0 \pmod{2N}}}\frac{1}{4n-t^2} H^2 \left(t^2-4n\right) \sum_{\substack{t^2<4n\\ t\equiv n_0 \pmod{2N}}}\left(4n-t^2\right) \\
\geq \left(\sum_{\substack{t^2<4n\\ t\equiv n_0 \pmod{2N}}} H \left(t^2-4n\right)\right)^2,
\end{multline*}
and so we have
\[
\sum_{\substack{t^2<4n\\t\equiv n_0 \pmod{2N}}}\frac{1}{4n-t^2} H^2 \left(t^2-4n\right) \gg  n^{-\frac{3}{2}} \left(\sum_{t^2<4n,~t\equiv n_0 \pmod{2N}} H \left(t^2-4n\right)\right)^2.
\]
Let $r_3\left(n\right)$ be the number of ways of representing $n$ as a sum of three squares. Then Gauss' formula (see for instance, \cite[Equation (1)]{KO}) asserts that
\[
r_3\left(n\right) =
\begin{cases}
12 H\left(-4n\right) & n\equiv 1,2 \pmod 4\\
 24 H\left(-n\right) & n\equiv 3 \pmod 8\\
r\left(\frac{n}{4}\right) & n\equiv 0 \pmod 4\\
0 & n\equiv 7 \pmod 8
\end{cases}.
\]
Observe from \eqref{hur} that if $4\nmid m$, then
\[
H\left(4^k m\right) = H\left(m\right) \left(\sigma_1\left(2^k\right) - \chi_D\left(2\right) \sigma_1\left(2^{k-1}\right)\right),
\]
and so
\[
2^kH\left(m\right) \leq H\left(4^k m\right) \leq  \left(2^{k+1}+2^k-2\right)H\left(m\right).
\]
Combining all these, we conclude that
\[
r_3\left(n\right) \leq 48 H\left(-n\right).
\]
Therefore we have
\[
48 \sum_{t^2<4n,~t\equiv n_0 \pmod{2N}} H \left(t^2-4n\right) \geq \sum_{t^2<4n,~t\equiv n_0 \pmod{2N}} r_3\left(4n-t^2\right),
\]
and observe that the last sum is equal to the number of elements in the following set:
\begin{equation}\label{class}
A_{2N}\left(n\right):=\{(x,y,z,t) \in \mathbb{Z}^4~:~4n=t^2+x^2+y^2+z^2,~ t\equiv n_0 \pmod{2N}\}.
\end{equation}
Note that we assume that $n$ is odd and $N$ is fixed. Kloosterman~\cite{Kloosterman} developed a version of the classical circle method with no minor arcs for quadratic forms in four variables. Based on the work of Kloosterman, we have~\cite[Theorem 1.6]{Sardari}
$$
A_{N}\left(n\right)\gg_N n.
$$
The work of the second author \cite[Theorem 1.6]{Sardari} gives the optimal exponent for  strong approximation for quadratic forms in five and more variables. For quadratic forms in four variables, it implies the above lower bound with an explicit dependence on $N$.

This completes the proof of the lower bound in Theorem \ref{arith}.

\subsection{Completion of proofs}
In this section, we prove Theorem \ref{selbergmain1}, \ref{selbergmain2}, and Corollary \ref{cor1}.
\begin{proof}[Proof of Theorem \ref{selbergmain1}]
This is a simple consequence of combining Lemma \ref{sufficient11} and Theorem \ref{arith}.
\end{proof}
\begin{proof}[Proof of Theorem \ref{selbergmain2}]
From Lemma \ref{sufficient11} and Theorem \ref{arith}, we see that the left-hand side of \eqref{selberg2} is
\[
> c_N\sqrt{n} - \frac{\sigma_1\left(n\right)^2}{An\sqrt{n}}
\]
for some constant $c_N>0$ depending only on $N$. If $n=p^m$, then $\sigma_1\left(n\right) = \frac{p^{m+1}-1}{p-1} < 2p^m = 2n$, which implies that
\[
c_N\sqrt{n} - \frac{\sigma_1\left(n\right)^2}{An\sqrt{n}} > \left(c_N- \frac{4}{A}\right)\sqrt{n}. \qedhere
\]
\end{proof}
\begin{proof}[Proof of Corollary \ref{cor1}]
We first note that from  \cite[(61)]{Gamburd} that for $n=p^m$,
\[
\left|\mathrm{Tr}~\mathcal{T}_n\left(k,N\right)^*- |B_{k,N}^*|\frac{\delta\left(n,\square\right)}{\sqrt{n}}\right|\leq 2m^2|B_{k,N}^*|D\left(\mu_{k,N}^*,\mu_p\right).
\]
By \eqref{dim} and Young's inequality $2x^2+2y^2 \geq \left(x+y\right)^2$,
\begin{multline*}
2\left|\mathrm{Tr}~\mathcal{T}_n\left(k,N\right)^*- |B_{k,N}^*|\frac{\delta\left(n,\square\right)}{\sqrt{n}}\right|^2   \\
\geq  \left|\mathrm{Tr}~\mathcal{T}_n\left(k,N\right)^*- \frac{k-1}{12} \varphi \left(N\right)\frac{\delta\left(n,\square\right)}{\sqrt{n}}\right|^2 + O\left(n^{-1}\right).
\end{multline*}
Now from Theorem \ref{selbergmain2}, we have
\begin{equation}\label{last}
\frac{1}{\sum_{k\in 2\mathbb{Z}} \phi\left(\frac{k-1}{K}\right) }\sum_{k>0, k\in 2\mathbb{Z}} \phi\left(\frac{k-1}{K}\right) m^4|B_{k,N}^*|^2D\left(\mu_{k,N}^*,\mu_p\right)^2 \gg_N n^{\frac{1}{2}},
\end{equation}
where $K = A\sqrt{n}$ for some fixed sufficiently large $A$. Suppose in order to obtain a contradiction that
\begin{equation}\label{llast}
D\left(\mu_{k,N}^*,\mu_p\right) = o\left(\frac{1}{k^{\frac{1}{2}}(\log k)^2}\right).
\end{equation}
Then from \eqref{last}, we have
\begin{align*}
n^{\frac{1}{2}}&\ll \frac{1}{\sum_{k\in 2\mathbb{Z}} \phi\left(\frac{k-1}{K}\right) }\sum_{k>0, k\in 2\mathbb{Z}} \phi\left(\frac{k-1}{K}\right) m^4|B_{k,N}^*|^2D\left(\mu_{k,N}^*,\mu_p\right)^2\\
&= o \left(\frac{1}{K}\sum_{k>0, k\in 2\mathbb{Z}} \phi\left(\frac{k-1}{K}\right) m^4\frac{k}{(\log k)^4}\right).
\end{align*}
However,
\[
\frac{1}{K}\sum_{k>0, k\in 2\mathbb{Z}} \phi\left(\frac{k-1}{K}\right) m^4\frac{k}{(\log k)^4} \ll m^4 \frac{K}{(\log K)^4} \ll \sqrt{n}
\]
contradicting the assumption \eqref{llast}.
\end{proof}

\section{Appendix: By Simon Marshall}

The purpose of this appendix is to illustrate the geometric origin of the transition behavior of the $J$-Bessel function, by recalling the derivation of the Petersson trace formula as a relative trace formula following \cite{KL}.  Let $G = PSL_2\left(\mathbb{R}\right)$, and $\Gamma = PSL_2\left(\mathbb{Z}\right)$.  Let $k \ge 2 $ be even, and define $f \in C^\infty\left(G\right)$ by
\[
f\left(g\right) = \frac{k-1}{4\pi} \frac{ \left(2i\right)^k}{ \left(-b+c + \left(a+d\right)i\right)^k}, \qquad g =  \begin{pmatrix} a& b \\ c& d \end{pmatrix}.
\]
This is the $L^2$-normalized matrix coefficient of the lowest weight vector in the weight $k$ discrete series \cite[Section 3.1]{KL}.  We form the function
\[
K_\Gamma\left(x,y\right) = \sum_{\gamma \in \Gamma} f\left(x^{-1} \gamma y\right)
\]
on $\left(\Gamma \backslash G\right)^2$.  The Petersson trace formula can be proved by integrating $K_\Gamma\left(x,y\right)$ against characters over two horocycles on $\Gamma \backslash G$, and comparing the geometric and spectral expansions of $K_\Gamma$.  More precisely, if $m, n \ge 1$ and we define
\[
\sigma_n = \begin{pmatrix} k / 4\pi n & \\ & 1 \end{pmatrix},
\]
and likewise for $\sigma_m$, then the integral we wish to expand is
\[
\int_0^1 \int_0^1 K_\Gamma\left( \begin{pmatrix} 1 & x \\ & 1 \end{pmatrix} \sigma_n, \begin{pmatrix} 1 & y \\ & 1 \end{pmatrix} \sigma_m \right) e\left(-nx + my\right) dx dy.
\]
Note that the heights we have chosen for our horocycles are optimal for picking up the $n$th and $m$th Fourier coefficients on the spectral side.

We shall analyze the geometric side of this integral, which is
\[
\int_0^1 \int_0^1 \sum_{\gamma \in \Gamma} f\left(\sigma_n^{-1} \begin{pmatrix} 1 & -x \\ & 1 \end{pmatrix} \gamma \begin{pmatrix} 1 & y \\ & 1 \end{pmatrix} \sigma_m \right) e\left(-nx + my\right) dx dy.
\]
We break the sum over $\gamma$ into double cosets $N \eta N$, which gives
\[
\sum_{\eta \in N \backslash \Gamma / N} \int_0^1 \int_0^1 \sum_{\gamma \in N \eta N} f\left(\sigma_n^{-1} \begin{pmatrix} 1 & -x \\ & 1 \end{pmatrix} \gamma \begin{pmatrix} 1 & y \\ & 1 \end{pmatrix} \sigma_m \right) e\left(-nx + my\right) dx dy.
\]
The contribution from the identity coset is
\[
\int_0^1 \int_0^1 \sum_{\gamma \in N} f\left(\sigma_n^{-1} \begin{pmatrix} 1 & -x \\ & 1 \end{pmatrix} \gamma \begin{pmatrix} 1 & y \\ & 1 \end{pmatrix} \sigma_m \right) e\left(-nx + my\right) dx dy.
\]
This vanishes unless $m = n$, in which case it is
\[
\frac{4\pi n}{k} \int_{-\infty}^\infty f\left( \begin{pmatrix} 1 & x \\ & 1 \end{pmatrix}  \right) dx,
\]
i.e., the integral of $f$ over the horocycle of height 1.  If $\eta \neq 1$, there is no repetition among the elements $n_1 \gamma n_2$, and so we may unfold the two integrals to obtain
\begin{equation}
\label{Ieta}
I_\eta = \int_{-\infty}^\infty \int_{-\infty}^\infty f\left(\sigma_n^{-1} \begin{pmatrix} 1 & -x \\ & 1 \end{pmatrix} \eta \begin{pmatrix} 1 & y \\ & 1 \end{pmatrix} \sigma_m \right) e\left(-nx + my\right) dx dy.
\end{equation}

This integral has a simple geometric meaning, as the integral of the kernel $K\left(x,y\right) = f\left(x^{-1}y\right)$ against characters over the two horocycles $N \sigma_n$ and $\eta N \sigma_m$.
If we write $\eta = \begin{pmatrix} a & b \\ c & d \end{pmatrix}$ with $c > 0$, then $c$ corresponds to the index of summation on the geometric side of the Petersson formula.  Moreover, the ranges $c < 4\pi \sqrt{mn}/k$, $c = 4\pi \sqrt{mn}/k$, and $c > 4\pi \sqrt{mn}/k$ correspond to the oscillation, transition, and decay range of the $J$-Bessel function in the following way.  We shall use the fact that the kernel $K$ concentrates near the diagonal in $\mathbb{H}^2 \times \mathbb{H}^2$.  If $c < 4\pi \sqrt{mn}/k$, then the two horocycles intersect transversally.  The integrand is roughly supported on two balls of radius $k^{-\frac{1}{2}}$ and has magnitude $k$, and we have $I_\eta \sim 1$ as expected.  If $c > 4\pi \sqrt{mn}/k$ then the horocycles do not intersect, and $I_\eta \ll_N k^{-N}$. The case $c = 4\pi \sqrt{mn}/k$ is where the horocycles are tangent, and so the integral is roughly supported on a ball of radius $k^{-\frac{1}{4}}$.  One might expect $I_\eta \sim k^{\frac{1}{2}}$ from this, but in fact it is of size $k^{\frac{1}{6}}$.  As we shall see below, the point is that the phase in \eqref{Ieta} has a cubic degeneracy, and this (rather than the support) determines the size of $I_\eta$.

We now explicate the relation between $I_\eta$ and the geometric side of the Petersson formula, and analyze the phase of the integral in the transition range.  Writing $\eta = \begin{pmatrix} a & b \\ c & d \end{pmatrix}$ with $c > 0$, the double coset $N \eta N$ is determined by $c$ and the residue class of $a$ mod $c$.  Moreover, we have
\[
\begin{pmatrix} a & b \\ c & d \end{pmatrix} = \begin{pmatrix} 1 & a/c \\ & 1 \end{pmatrix} \begin{pmatrix} & -1/c \\ c & \end{pmatrix} \begin{pmatrix} 1 & d/c \\ & 1 \end{pmatrix}.
\]
Changing variable in $x$ and $y$ by a translation, we have
\begin{multline*}
I_\eta = e\left( -\left(na + md\right)/c\right)\\
\times \int_{-\infty}^\infty \int_{-\infty}^\infty f\left(\sigma_n^{-1} \begin{pmatrix} 1 & -x \\ & 1 \end{pmatrix} \begin{pmatrix} & -1/c \\ c & \end{pmatrix} \begin{pmatrix} 1 & y \\ & 1 \end{pmatrix} \sigma_m \right) e\left(-nx + my\right) dx dy.
\end{multline*}
Conjugating the matrices $\sigma_n$ and $\sigma_m$ though to the middle and changing variable gives
\begin{multline*}
I_\eta = e\left( -\left(na + md\right)/c\right) \frac{k^2}{\left(4 \pi\right)^2 mn} \\
\int_{-\infty}^\infty \int_{-\infty}^\infty f\left( \begin{pmatrix} 1 & -x \\ & 1 \end{pmatrix} \begin{pmatrix} & -4 \pi n / kc \\ kc / 4 \pi m & \end{pmatrix} \begin{pmatrix} 1 & y \\ & 1 \end{pmatrix} \right)
e\left(k\left(-x+y\right)/4\pi\right) dx dy.
\end{multline*}
If we define
\[
A\left(t,k\right) = \int_{-\infty}^\infty \int_{-\infty}^\infty f\left( \begin{pmatrix} 1 & -x \\ & 1 \end{pmatrix} \begin{pmatrix} & -1/t \\ t & \end{pmatrix} \begin{pmatrix} 1 & y \\ & 1 \end{pmatrix} \right) e\left(k\left(-x+y\right)/4\pi\right) dx dy,
\]
then the contribution from all $\eta$ with a given value of $c$ is
\[
\frac{k^2}{\left(4 \pi\right)^2 mn} S\left(m,n,c\right) A\left(kc / 4 \pi \sqrt{mn}, k\right).
\]
In \cite[Prop. 3.6]{KL}, Knightly and Li calculate
\[
A\left(t,k\right) = \frac{e^{-k} i^k 4 \pi k^{k-1} }{ 2t \left(k-2\right)!} J_{k-1}\left( k /t\right) \sim \frac{k^{\frac{1}{2}}}{t} J_{k-1}\left( k /t\right),
\]
which gives the required appearance of $J_{k-1}$ on the geometric side.

One again sees the geometric meaning of $A\left(t,k\right)$.  It is an integral of $K\left(x,y\right)$ against characters over a horocycle of height 1, and a horocycle corresponding to the point $0 \in \partial \mathbb{H}^2$ and whose highest point is at $i / t^2$.  One therefore expects a transition of $A\left(t,k\right)$ at $t = 1$, and this corresponds to $c = 4\pi \sqrt{mn} / k$ as claimed above.  We now write $A\left(1,k\right)$ as an oscillatory integral (with non-imaginary phase function), and examine its critical point.  Using our formula for $f$ gives
\begin{align*}
f\left( \begin{pmatrix} 1 & -x \\ & 1 \end{pmatrix} \begin{pmatrix} & -1 \\ 1 & \end{pmatrix} \begin{pmatrix} 1 & y \\ & 1 \end{pmatrix} \right)& = f\left( \begin{pmatrix} -x & -1 -xy \\ 1 & y \end{pmatrix} \right) \\
& = \frac{k-1}{4\pi} i^k \left(1 + \frac{xy}{2} + i\frac{\left(y-x\right)}{2}\right)^{-k} \\
& = \frac{k-1}{4\pi} i^k \exp\left( -k \log\left(1 + \frac{xy}{2} + i\frac{\left(y-x\right)}{2}\right) \right).
\end{align*}
Computing the Taylor expansion of $\log\left(1 + \frac{xy}{2} + i\frac{\left(y-x\right)}{2}\right)$ gives
\begin{multline*}
\log\left(1 + \frac{xy}{2} + i\frac{\left(y-x\right)}{2}\right) \\
 = \frac{xy}{2} + i\frac{\left(y-x\right)}{2} -\frac{1}{2}\left( -\frac{\left(y-x\right)^2}{4} + i \frac{xy \left(y-x\right)}{2}\right)
 - 4i \left(y-x\right)^\frac{3}{2} + O\left(x^4 + y^4\right) \\
 = \frac{\left(x+y\right)^2}{8} + i\left( \frac{\left(y-x\right)}{2} -\frac{xy\left(y-x\right)}{4} -4\left(y-x\right)^\frac{3}{2} \right) + O\left(x^4 + y^4\right).
\end{multline*}
Substituting this into $A\left(1,k\right)$ gives
\begin{multline*}
A\left(1,k\right) = \frac{k-1}{4\pi} i^k \times \\
\iint_{\mathbb{R}^2} \exp\left( -k\frac{\left(x+y\right)^2}{8} + ik\left( \frac{xy\left(y-x\right)}{4} + \frac{\left(y-x\right)^3}{ 24}\right)  + k O\left(x^4 + y^4\right) \right) dx dy.
\end{multline*}
The leading term $-k\left(x+y\right)^2/8$ in the phase truncates the integral to the line $x+y = 0$ at scale $k^{-\frac{1}{2}}$, and along this line the leading term in the phase is imaginary with a cubic degeneracy.  This is why one has $A\left(1,k\right) \sim k^{\frac{1}{6}}$ compared to $A\left(t,k\right) \sim 1$ for $t < 1$.

\bibliography{20200529final}
\bibliographystyle{alpha}

\end{document}